\documentclass[11pt]{article}
\usepackage[letterpaper, margin=1.0in]{geometry}
\usepackage{bm}
\usepackage{xcolor}
\usepackage{setspace}

\makeatletter

\usepackage[sort&compress,numbers]{natbib}
\usepackage{appendix}
\usepackage{enumerate}
\usepackage[ruled,linesnumbered]{algorithm2e}
\usepackage{multirow}
\usepackage{graphicx}
\usepackage{subcaption}
\usepackage{tabularx}
\usepackage{makecell}
\usepackage{booktabs}

\usepackage{comment}

\usepackage[ruled]{algorithm2e}
\SetKwInOut{Input}{Input}
\SetKwInOut{Output}{Output}

\usepackage{amsmath}
\usepackage{amstext}
\usepackage{amsthm}
\usepackage{amssymb}
\usepackage{bbm}
\usepackage{amsbsy}
\usepackage{mathtools}

    \newtheorem{prop}{Proposition}[section]
    \numberwithin{prop}{section} 
    \newtheorem{lem}{Lemma}[section]
    \numberwithin{lem}{section} 
    \newtheorem{rem}{Remark}[section]
    \numberwithin{rem}{section} 
    
    \numberwithin{cor}{section} 
    \newtheorem{thm}{Theorem}[section]
    \numberwithin{thm}{section} 
    \newtheorem{defn}{Definition}
    \numberwithin{defn}{section}
    
    \numberwithin{fact}{section}

    \theoremstyle{plain}
    \newtheorem*{thm*}{Theorem}
    \newtheorem*{rem*}{Remark}
    
    \newcommand{\beq}{\begin{equation}}
    \newcommand{\eeq}{\end{equation}}
    \newcommand{\beqnn}{\begin{equation*}}
    \newcommand{\eeqnn}{\end{equation*}}
        
    \def\bb {\boldsymbol{b}}
    \def\bd {\boldsymbol{d}}
    
    \def\bdelta {\boldsymbol{\delta}}
    
    \def\tzero {t_{0}}
    
    \def\be {\boldsymbol{e}}

    \def\bp {\boldsymbol{p}}
    \def\bq {\boldsymbol{q}}
    \def\bu {\boldsymbol{u}}
    \def\bs {\boldsymbol{s}}
    \def\bv {\boldsymbol{v}}
    
    \def\bx {\boldsymbol{x}}
    \def\by {\boldsymbol{y}}

    \def\R {\mathbb{R}}
    
    \def\bxi {\boldsymbol{\xi}}

    \newcommand{\normsq}[1]{\left\Vert{#1} \right\Vert_{2}^{2}}
    \newcommand{\normtwo}[1]{\left\Vert{#1} \right\Vert_{2}}
    \newcommand{\normone}[1]{\left\Vert{#1} \right\Vert_{1}}
    
    \newcommand{\norminf}[1]{\left\Vert{#1} \right\Vert_{\infty}}
    
    \newcommand{\matr}[1]{\bm{#1}}

    \def\cone{\mathrm{cone}}
    
    \def\dom {\mathrm{dom~}}    
    \def\interior {\mathrm{int }}  
    \def\ri {\mathrm{ri}}      
    \def\sign {\mathrm{sgn}}
    \def\diag {\mathrm{diag}}
    
    \DeclareMathAlphabet{\mymathbb}{U}{BOONDOX-ds}{m}{n}

    \def\Binf {\mathrm{B}_{\infty}}
    \def\proj {\mathrm{proj}}
    
    \def\Rn {\R^n}
    \def\Rm {\R^m}
    
    \def\gmRn {\Gamma_0(\Rn)}   
    \def\gmRm {\Gamma_0(\Rm)}

    \DeclareMathOperator*{\argmin}{arg\,min}

    \def\bxsol {\bx^{s}}
    \def\bpsol {\bp^{s}}
    \newcommand{\image}[1]{\mathrm{Im}({#1})}
    \newcommand{\rank}[1]{\mathrm{rank}({#1})}
    \def\bpo {\bp_{0}}
    \def\bpa {\bp_{a}}

    \def\eqset {\mathcal{E}}
    \def\eqcset {\mathcal{E}^{\mathsf{C}}}

    \def\ep {\mathcal{E}(\bp)}

    \def\epdo {\mathcal{E}(\bpo + \Delta\bd(\bpo,t))}
    \def\ecpdo {\mathcal{E}^{\mathsf{C}}(\bpo + \Delta\bd(\bpo,t))}

    \def\epo {\mathcal{E}(\bp_{0})}
    \def\ecpo {\mathcal{E}^{\mathsf{C}}(\bp_{0})}

    \def\bvlsq {\hat{\bv}^{\text{LSQ}}}

    \def\Atop {\matr{A}^{\top}}

\usepackage{hyperref}

\title{A fast algorithm for solving the lasso problem exactly without homotopy using differential inclusions} 
\author{Gabriel P. Langlois$^1$ \and J\'er\^ome Darbon$^2$}
\date{%
    $^1$Department of Mathematics, University of Illinois Urbana-Champaign (\url{gabriel_provencher_langlois@alumni.brown.edu}). \textbf{Corresponding author}.\\%
    $^2$Division of Applied Mathematics, Brown University (\url{jerome_darbon@brown.edu}).\\[2ex]%
    \today
}

\begin{document}
\maketitle
\begin{abstract}
We prove in this work that the well-known lasso problem can be solved exactly without homotopy using novel differential inclusions techniques. Specifically, we show that a selection principle from the theory of differential inclusions transforms the dual lasso problem into the problem of calculating the trajectory of a projected dynamical system that we prove is integrable. Our analysis yields an exact algorithm for the lasso problem, numerically up to machine precision, that is amenable to computing regularization paths and is very fast. Moreover, we show the continuation of solutions to the integrable projected dynamical system in terms of the hyperparameter naturally yields a rigorous homotopy algorithm. Numerical experiments confirm that our algorithm outperforms the state-of-the-art algorithms in both efficiency and accuracy. Beyond this work, we expect our results and analysis can be adapted to compute exact or approximate solutions to a broader class of polyhedral-constrained optimization problems.
\end{abstract}

\noindent \textit{Keywords:} Differential inclusions, projected dynamical systems, optimization, exact solutions, algorithms, lasso, basis pursuit denoising, compressive sensing, machine learning, inverse problems.

\medskip
\noindent \textbf{Mathematics Subject Classification: 90C25, 65K05, 37N40, 46N10, 34A60, 62J07}

\section{Introduction}\label{sec:intro}
The lasso problem is a cornerstone to many high-dimensional applications in, e.g., statistics, machine learning, compressive sensing, and inverse problems~\cite{chen2001atomic,tibshirani1996regression,nikolova2002minimizers,vossen2006l1,stadler2009elliptic}. Also known as basis pursuit denoising, the (constrained) lasso problem is given by\begin{equation}\label{eq:lasso_primal}\tag{LASSO}
\min_{\bx \in \Rn} \left\{\normone{\bx} + \frac{1}{2t}\normsq{\matr{A}\bx - \bb}\right\},
\end{equation}
where $\matr{A}$ is a real $m \times n$ matrix with $m \leqslant n$, $\bb \in \Rm$ is the observed data, and $t \geqslant 0$ is a hyperparameter controlling the trade-off between sparsity and data fidelity. The limit $t \to 0$ yields the limiting problem known in signal processing as basis pursuit:
\begin{equation}\label{eq:BP_primal}\tag{BP}
\min_{\bx \in \Rn} \normone{\bx} \quad \text{subject to} \quad \matr{A}\bx = \bb,
\end{equation}
The interpretation of $\matr{A}$ and $\bb$ depends on the application, while the hyperparameter $t$ is either preselected or estimated with data-driven methods such as cross-validation.

Estimating the appropriate hyperparameter can prove challenging in practice. The standard approach is to construct a \emph{regularization path}. To do so, one first selects a sequence $\{t^{(k)}\}_{k=0}^{K}$, one then computes the corresponding solutions $\{\bx^{s}(t^{(k)},\bb)\}_{k=0}^{K}$ to~\eqref{eq:lasso_primal}, and one finally chooses the hyperparameter that gives the preferred solution. Different methods differ in how they select the hyperparameters and compute the solutions, but in any case, constructing a regularization path entails solving~\eqref{eq:lasso_primal} accurately for many hyperparameters. This approach can therefore become time-consuming and computationally intensive when the dimensions $m$ and $n$ are high. This issue has stimulated significant research to develop efficient and accurate algorithms for solving~\eqref{eq:lasso_primal} in high dimensions~\cite{osborne2000new,osborne2000lasso,efron2004least,daubechies2004iterative,friedman2007pathwise,donoho2008fast,beck2009fast,bringmann2018homotopy}.

While many algorithms have been proposed for the lasso, they invariably suffer from drawbacks and must either favor efficiency over accuracy or vice versa. State-of-the-art algorithms therefore remain ineffective for high-dimensional applications requiring accurate solutions in a reasonable amount of computational time. We address this issue in this work and present a fast algorithm for solving the lasso problem exactly using differential inclusions. Our analysis and results yield Algorithm~\ref{alg:lasso}. It computes an exact pair of primal and dual solutions to the lasso problem, numerically up to machine precision, is amenable to computing regularization paths, and is very fast.

\paragraph{Contributions of this paper:} \textbf{(i)} We prove that a selection principle from the theory of differential inclusions turns the dual lasso problem into calculating the trajectory of an \emph{integrable} projected dynamical system, which we then calculate exactly. Our main results, which are presented in Section~\ref{sec:results}, culminate in Algorithm~\ref{alg:lasso}, an exact algorithm that computes the optimal primal and dual lasso solutions for any $t\geqslant0$ without homotopy. As a by-product, our results provide a novel solution method for solving a broad class of projected dynamical systems, which should find relevance to applications outside the scope of this work. \textbf{(ii)} We present in Section~\ref{sec:perturbations} a detailed continuation analysis of solutions of the projected dynamical system in terms of the hyperparameter $t$, thereby yielding a rigorous, generalized homotopy algorithm for the lasso problem. \textbf{(iii)} Our numerical experiments show that Algorithm~\ref{alg:lasso} vastly outperforms the state of the arts in accuracy while also achieving the best overall performance, highlighting a key feature of Algorithm~\ref{alg:lasso}: it neither compromises accuracy nor computational efficiency.

\paragraph{Related work:} State-of-the-art algorithms for the lasso are divided roughly in three categories: coordinate descent methods~\cite{friedman2007pathwise,wu2008coordinate,friedman2010regularization,tibshirani2012strong}, first-order optimization algorithms such as FISTA or the PDHG~\cite{glowinski1975approximation,yin2008bregman,beck2009fast,chambolle2011first,chambolle2016introduction}, and homotopy algorithms~\cite{osborne2000new,osborne2000lasso,efron2004least,donoho2008fast,tibshirani2013lasso,bringmann2018homotopy}. We will focus only on these algorithms; other algorithms include the iteratively reweighted least squares algorithm~\cite{daubechies2004iterative}, Bayesian methods~\cite{park2008bayesian}, adaptive inverse scale space methods~\cite{burger2013adaptive}, specialized (quasi-) Newton methods~\cite{karimi2017imro}, fixed-point continuation methods~\cite{hale2008fixed}, and interior point methods~\cite{boyd2004convex}. See~\cite{zhao2023survey,bertsimas2020,li2020survey} for surveys and comparisons of different approaches. Coordinate descent methods are the state of the arts because they are fast. Key to their efficiency are so-called selection rules or heuristics~\cite{tibshirani2012strong,ghaoui2010safe,raj2016screening} that estimate, \textit{a priori}, the degree of sparsity of primal solutions. However, these methods suffer from algorithmic instability~\cite{barber2021stability} and may therefore produce inaccurate numerical solutions. First-order optimization algorithms, in contrast, are numerically accurate but less efficient in high dimensions. Finally, homotopy algorithms compute exact solutions paths to~\eqref{eq:lasso_primal} but generally require technical assumptions to work, e.g., the uniqueness of the path and the ``one-at-a-time condition"~\cite{tibshirani2013lasso,efron2004least,osborne2000new}. These assumptions are difficult to verify and may not hold in practice. For example, the LARS algorithm fails to converge on simple examples~\cite[Proposition 4.1]{bringmann2018homotopy}. Moreover, all homotopy algorithms have exponential worst-case complexity~\cite{mairal2012complexity}, and while in practice they often converge fast, they are generally less efficient than coordinate descent methods.

\begin{algorithm}[ht]
    \caption{Algorithm for computing a pair of primal and dual lasso solutions.}
    \label{alg:lasso}
    \Input{A matrix $\matr{A} \in \mathbb{R}^{m \times n}$, a vector $\bb \in \image{\matr{A}}\setminus{\{\boldsymbol{0}\}}$ and a number $t\geqslant 0$.}
    \Output{A pair of primal and dual solutions $(\bx^{s}(t,\bb),\bp^{s}(t,\bb)) \in \Rn \times \Rm$ to the lasso problem.}
    \BlankLine
    \emph{Set $\bp^{(0)} \in \Rm$ such that $\norminf{\Atop\bp^{(0)}} = 1$}\;
    \For{$k = 1$ until convergence} {
    	Compute $\eqset^{(k-1)} = \{j \in \{1,\dots,n\} : |\langle-\Atop\bp^{(k-1)},\be_{j}\rangle| = 1\}$\;
    	Compute $\matr{D}^{(k-1)} = \diag\left(\sign{(-\Atop\bp^{(k-1)})}\right)$\;
        Compute $\hat{\bu}^{(k-1)} \in \argmin_{\substack{\bu_{\eqset^{(k-1)}} \geqslant \boldsymbol{0} \\ \bu_{{(\mathcal{E}^{(k-1)})^\mathsf{C}}} = \boldsymbol{0}} } \normsq{\matr{A}\matr{D}^{(k-1)}\bu - \bb - t\bp^{(k-1)}}$ \;
        Compute the descent direction $\bd^{(k-1)} = \matr{A}\matr{D}^{(k-1)}\hat{\bu}^{(k-1)} - \bb - t\bp^{(k-1)}$\;
        Compute the maximal descent time
        \[
        \Delta^{(k-1)} = \min_{j \in \{1,\dots,n\}} \left\{\frac{\sign{\left\langle \matr{D}^{(k-1)}\Atop\bd^{(k-1)},\be_j\right\rangle} - \left\langle \matr{D}^{(k-1)}\Atop\bp^{(k-1)},\be_j\right\rangle)}{\left\langle\matr{D}^{(k-1)} \Atop\bd^{(k-1)},\be_j\right\rangle} \right\}\;
        \]
        
        \uIf{$t>0$ and $t\Delta^{(k-1)} \geqslant 1$}{ 
            Set $\bx^{s}(t,\bb) = \matr{D}^{(k-1)}\hat{\bu}^{(k-1)}$; \\
            Set $\bp^{s}(t,\bb) = \bp^{(k-1)} + \bd^{(k-1)}/t$; \\
            break; \tcp*[h]{The algorithm has converged}
        }
        \uElseIf{$t = 0$ and $\bd^{(k-1)} = \boldsymbol{0}$}{
            Set $\bx^{s}(t,\bb) = \matr{D}^{(k-1)}\hat{\bu}^{(k-1)}$; \\
            Set $\bp^{s}(t,\bb) = \bp^{(k-1)}$; \\
        	break; \tcp*[h]{The algorithm has converged}
        }
        Update $\bp^{(k)} = \bp^{(k-1)} + \Delta^{(k-1)}\bd^{(k-1)}$\;
    }
\end{algorithm}

\paragraph{Organization of this paper:} We review in Section~\ref{sec:setup} the existence of solutions and optimality conditions to the lasso problem and introduce a characterization in terms of differential inclusions. In Section~\ref{sec:msp}, we present the minimal selection principle and use it to cast the dual lasso problem into the equivalent problem of computing the trajectory of a projected dynamical system. Section~\ref{sec:results} characterizes the projected dynamical system and shows that it can be integrated exactly, and in particular that its trajectory and asymptotic limit can be computed explicitly. This gives the optimal solution to the dual lasso problem when $t>0$ (an optimal solution when $t=0$) and recover an optimal primal solution. We present in Section~\ref{sec:perturbations} a detailed continuation analysis of the asymptotic limit of the projected dynamical system, i.e., solutions to the lasso problem, in terms of the hyperparameter $t$, yielding a generalized homotopy algorithm. We present numerical experiments in Section~\ref{sec:numerics} to compare our algorithms to some state-of-the-art algorithms for the lasso problem, and finally we discuss the broader implications of our work in Section~\ref{sec:discussion}. 
\section{Setup}\label{sec:setup}
Unless stated otherwise, we assume $\bb \in \text{Im}(\matr{A})\setminus\{\boldsymbol{0}\}$ and $\rank{\matr{A}} = m$. The case $\bb = \boldsymbol{0}$ is uninteresting and we can assume the latter because, otherwise, at least one row of the matrix $\matr{A}$ is linearly dependent on the others and can be discarded. 

We wish to note that some analyses and proofs in this paper are fairly involved and use concepts (e.g., from convex analysis, functional analysis and dynamical systems) that may be unfamiliar to the reader. We refer the reader to Appendix~\ref{app:A} for more detailed mathematical background.


\subsection{Existence of solutions and optimality conditions to the lasso problem}\label{subsec:reviewI}
We review here existence and optimality results using classic results from convex analysis, summarized as Theorem~\ref{thm:primal_dual_prob} in Appendix~\ref{app:A}. Following the notation of Theorem~\ref{thm:primal_dual_prob}, we set 
\[
f_{1}(\cdot) = \normone{\cdot} \quad \text{and} \quad f_{2}(\cdot) = \begin{cases}
\frac{1}{2t}\normsq{\cdot-\bb}\, &\text{if } t > 0 \\
\chi_{\{b\}}(\cdot)\, &\text{if } t = 0
\end{cases}
\]
where $\chi_{\bb}(\cdot)$ denotes the characteristic function of the singleton set $\{\bb\}$. Since $\dom f_{1} = \Rn$, $\bb \in \ri \, \dom f_{2}$ and $\bb \in \image{\matr{A}}$, there exists $\bx \in \ri \, \dom f_{1}$ such that $\matr{A}\bx \in \ri \, \dom f_{2}$. Moreover, the function $\bx \mapsto f_{1}(\bx) + f_{2}(\matr{A}\bx)$ is coercive because $f_{1}$ is coercive and $f_{2}$ is nonnegative. Using Remark~\ref{rem:coercivity}, we conclude that all conditions of Theorem~\ref{thm:primal_dual_prob} are satisfied. As a consequence,~\eqref{eq:lasso_primal} and~\eqref{eq:BP_primal} both have at least one solution. The dual lasso problem is given by
\begin{equation}\label{eq:lasso_dual}\tag{dLASSO}
    -\inf_{\bp \in \Rm} \left\{\frac{t}{2}\normsq{\bp} + \langle\bp,\bb\rangle + \chi_{\Binf}(-\Atop\bp)\right\}
\end{equation}
and the limit $t\to 0$ yields the limiting problem dual basis pursuit problem:
\begin{equation}\label{eq:BP_dual}\tag{dBP}
    -\inf_{\bp \in \Rm} \left\{\langle\bp,\bb\rangle + \chi_{\Binf}(-\Atop\bp)\right\}.
\end{equation}
Here, $\chi_{B_{\infty}}(\cdot)$ denotes the characteristic function of the unit $\ell_{\infty}$-ball:
\begin{equation}\label{eq:charfunl1}
\chi_{\Binf}(\bs) = \begin{cases}
0\, &\, \text{if } |\langle \bs,\be_{j}\rangle| \leqslant 1 \,\text{for every}\, j\in\{1,\dots,n\},\\
+\infty\, &\, \text{otherwise}.
\end{cases}
\end{equation}
Problem~\eqref{eq:lasso_dual} has a unique solution for every $t>0$ due to strong convexity and problem~\eqref{eq:BP_dual} has at least one solution. Moreover,~\eqref{eq:lasso_dual} and~\eqref{eq:BP_dual} are equal in value to their respective primal problems~\eqref{eq:lasso_primal} and~\eqref{eq:BP_primal}. Finally, letting $(\bxsol(t,\bb),\bpsol(t,\bb))$ denote any pair of solutions to the lasso problem and its dual, we have the set of equivalent first-order optimality conditions:
\begin{subequations}\label{eq:bpdn_all_oc}
\begin{align}
-t\bpsol(t,\bb) &= \bb - \matr{A}\bxsol(t,\bb) \quad \text{and} \ -\matr{A}^{\top}\bpsol(t,\bb) \in \partial\normone{\cdot}(\bxsol(t,\bb)), \label{eq:bpdn_mixed_oc} \\
-t\bpsol(t,\bb) &\in  \bb - \matr{A}\partial\chi_{\Binf}(-\matr{A}^{\top}\bpsol(t,\bb)),\label{eq:lasso_dual_oc}
\end{align}
\end{subequations}
where $\lim_{t\to 0} t\bpsol(t,\bb) = 0$, with $\matr{A}\bxsol(0,\bb) = \bb$ and $-\matr{A}^{\top}\bp^{s}(0,\bb) \in \partial \normone{\cdot}(\bxsol(0,\bb))$.

\subsection{Structure of the optimality conditions}\label{subsec:reviewII}
The optimality conditions in~\eqref{eq:bpdn_mixed_oc} on the right identify the components of $-\matr{A}^{\top}\bpsol(t,\bb)$ achieving maximum absolute deviation:
\begin{equation}\label{eq:equicorrelation_char}
    \langle-\matr{A}^{\top}\bpsol(t,\bb),\be_{j}\rangle =
    \begin{cases}
    1\, &\, \text{if } x_{j}^{s}(t,\bb) > 0 \\
    -1\, &\, \text{if } x_{j}^{s}(t,\bb) < 0 \\
    [-1,1]\, &\, \text{if } x_{j}^{s}(t,\bb) = 0.
    \end{cases}
\end{equation}
In particular, $|\langle-\matr{A}^{\top}\bpsol(t,\bb),\be_{j}\rangle| < 1 \implies x_{j}^{s}(t,\bb) = 0$. It will therefore be useful to identify for any $\bp \in \Rm$ the set of indices $j \in \{1,\dots,n\}$ for which $|\langle-\matr{A}^{\top}\bp,\be_{j}\rangle| = 1$. This set is called the \emph{equicorrelation set} at $\bp$ and we will denote it by $\mathcal{E}(\bp)$:
\begin{equation}\label{eq:equicorrelation_set}
    \eqset(\bp) \coloneqq \left\{j \in \{1,\dots,n\} : |\langle-\matr{A}^{\top}\bp,\be_{j}\rangle| = 1 \right\}.
\end{equation}
In addition, we will keep track of the signs and define the $n \times n$ diagonal matrix of signs
\begin{equation}\label{eq:matrix_signs}
\matr{D}(\bp) = \diag\left(\sign{(-\Atop\bp)}\right).
\end{equation}
The equicorrelation set identifies the active constraints of the cone~\eqref{eq:lasso_dual_oc}. To see this, observe the $\ell_{\infty}$ unit ball is a closed convex polyhedron, and so Proposition~\eqref{prop:subdiff_ccp} implies
\[
\bb - \matr{A}\partial\chi_{\Binf}(-\matr{A}^{\top}\bp) = \left\{\bb - \sum_{j \in \ep} u_{j} \matr{A}\matr{D}(\bp)\be_{j} : u_{j} \geqslant 0\right\}
\]
for all $\bp \in \dom V(\cdot;t,\bb)$. The optimality conditions~\eqref{eq:lasso_dual_oc} therefore read
\begin{equation}\label{eq:bpdn_full_oc}
-t\bpsol(t,\bb) \in \left\{\bb - \sum_{j \in \mathcal{E}(\bpsol(t,\bb))} u_{j} \matr{A}\matr{D}(\bpsol(t,\bb))\be_{j} : u_{j} \geqslant 0\right\}.
\end{equation}

\subsection{Characterization via differential inclusions}\label{differential inclusions}
We now present a characterization of solutions to the dual problem~\eqref{eq:lasso_dual} in terms of an initial value problem involving differential inclusions. Let $V \colon \Rm \times [0,+\infty) \times \Rm \to \R \cup \{+\infty\}$ denote the objective function in~\eqref{eq:lasso_dual}, parameterized in terms of $t$ and $\bb$:
\begin{equation}\label{eq:BPDN_potential}
    V(\bp;t,\bb) = \frac{t}{2}\normsq{\bp} + \langle\bp,\bb\rangle + \chi_{\Binf}(-\Atop\bp).
\end{equation}
Its subdifferential with respect to $\bp$ is the convex cone 
\begin{equation}\label{eq:BPDN_subdiff_potential}
    \partial_{\bp} V(\bp;t,\bb) = \left\{t\bp + \bb - \sum_{j \in \mathcal{E}(\bp)} u_{j} \matr{A}\matr{D}(\bp)\be_{j} : u_{j} \geqslant 0\right\}.
\end{equation}
(See Proposition~\ref{prop:subdiff_ccp} for details.) We suggest to compute a solution to the dual problem using a nonsmooth generalization of gradient descent on $\bp \mapsto V(\bp;t,\bb)$:
\begin{equation}\label{eq:V_diff_inclusion}
    \dot{\bp}(\tau) \in - \partial_{\bp} V(\bp(\tau);t,\bb), \qquad \bp(0) = \bp_{0} \in \dom \partial_{\bp} V(\cdot; t,\bb).
\end{equation}

Differential inclusions generalize the concept of ordinary differential equations to multi-valued \emph{maximal monotone mappings}~\cite{aubin2012differential,brezis1973ope}. As a special case, these mappings include subdifferentials of proper, lower semicontinuous and convex functions~\cite[Theorem 12.17]{rockafellar2009variational}. Their corresponding differential inclusions are called \emph{gradient inclusions}~\cite{broer2010handbook} because they generalize the classical concept of gradient systems. Indeed, similarly to how trajectories of gradient systems converge to their critical points (if any exist), trajectories of gradient inclusions converge to their critical points (if any exist). This fact follows from a variational principle called the \emph{minimal selection principle}.
\section{The minimal selection principle}\label{sec:msp}
The minimal selection principle stipulates that solutions to gradient inclusions exist, are unique, and evolve in the steepest descent direction~\cite[Chapter 3]{aubin2012differential}. As the minimal selection principle is central to this work, we state it below in full:
\begin{thm}[Existence and uniqueness of solutions to gradient inclusions]\label{thm:eu_diff_inclusions} Let $g\colon \Rm \to \R\cup\{+\infty\}$ be a proper, lower semicontinuous, and convex function and let $\bp_{0} \in \dom \partial g$. Consider the gradient inclusions
\begin{equation}\label{eq:general_diff_inclusion}
    \dot{\bp}(\tau) \in - \partial g(\bp(\tau)), \qquad \bp(0) = \bp_{0}.
\end{equation}
Then, there exists a unique solution $\bp \colon [0,+\infty) \mapsto \dom \partial g$ satisfying~\eqref{eq:general_diff_inclusion}. Moreover:
\begin{itemize}
    \item[(i)] The function $\tau \mapsto \dot{\bp}(\tau)$ is right-continuous and the function $\tau \mapsto \normtwo{\dot{\bp}(\tau)}$ is nonincreasing.
    
    \item[(ii)] If $g$ achieves its minimum at some point, then $\bp(\cdot)$ converges to such a point:
    \[
    \lim_{\tau \to +\infty} g(\bp(\tau)) = \min_{\bp \in \Rm} g(\bp) \quad \text{and} \quad \lim_{\tau \to +\infty} \bp(\tau) \in \argmin_{\bp \in \Rm} g(\bp).
    \]
    \item[(iii)] (The minimal selection principle.) The function $\bp(\cdot)$ satisfies the initial value problem
    \begin{equation}\label{eq:general_slow_ode}
            \dot{\bp}(\tau) =-\proj_{\partial g(\bp(\tau))}(\boldsymbol{0}), \qquad \bp(0) = \bp_{0}
    \end{equation}
    at $\tau = 0$ and almost everywhere on $(0,+\infty)$. 
\end{itemize}
\end{thm}
\begin{proof}
    See~\cite[Theorem 1, page 147 and Theorem 1, page 159]{aubin2012differential} for the proofs of (i)--(iii).
\end{proof}

\begin{rem}
In the language of dynamical systems theory, the initial value problem~\eqref{eq:general_slow_ode} is called a \emph{projected dynamical system}~\cite{dupuis1993dynamical}. Such systems arise naturally in the theory of variational inequalities~\cite[Chapter 17]{attouch2014variational}, as was first noted in~\cite{dupuis1991lipschitz}.
\end{rem}


\subsection{The minimal selection principle for the lasso problem}\label{subsec:selection_principle}

In the context of this work, the minimal selection principle implies the system of gradient inclusions~\eqref{eq:V_diff_inclusion} has a unique, global solution satisfying the initial value problem
\begin{equation}\label{eq:bp_slow_ode}
    \dot{\bp}(\tau) =-\proj_{\partial_{\bp} V(\bp(\tau);t,\bb)}(\boldsymbol{0}), \qquad \bp(0) \in \dom \partial_{\bp} V(\cdot;t,\bb)
\end{equation}
at $\tau = 0$ and almost everywhere on $(0,+\infty)$, where
\[
\proj_{\partial_{\bp} V(\bp(\tau);t,\bb)}(\boldsymbol{0}) = \argmin_{\bs \in \Rm} \normsq{\bs} \ \text{such that} \ \bs \in \partial_{\bp}V(\bp(\tau);t,\bb).
\]
Moreover, the solution converges asymptotically to the unique solution of~\eqref{eq:lasso_dual} when $t>0$ (a solution of~\eqref{eq:BP_dual} when $t=0$). The projected dynamical system~\eqref{eq:bp_slow_ode} is called the \emph{slow system} and its solution the \emph{slow solution}. In addition, we will write
\begin{equation}\label{eq:steepest_dd}
\bd(\bp(\tau);t,\bb) \coloneqq -\proj_{\partial_{\bp} V(\bp(\tau);t,\bb)}(\boldsymbol{0})
\end{equation}
to denote the direction of change of the slow system at $\bp(\tau)$.

The minimal selection principle suggests we can compute a solution to~\eqref{eq:lasso_dual} or~\eqref{eq:BP_dual} by calculating the asymptotic limit of the slow system~\eqref{eq:bp_slow_ode}. To do this, we must compute the minimal selection $\proj_{\partial_{\bp} V(\bpo;t,\bb)}(\boldsymbol{0})$ for any $\bpo \in \dom \partial_{\bp}V(\cdot,t,\bb)$. We turn to this problem next.


\subsection{Computing the minimal selection}\label{subsec:computing_ms}
Here, we describe how the minimal selection $\proj_{\partial_{\bp} V(\bpo;t,\bb)}(\boldsymbol{0}) \equiv -\bd(\bpo;t,\bb)$ of the slow system~\eqref{eq:bp_slow_ode} is computed from a cone projection problem or, equivalently, a nonnegative least-squares (NNLS) problem. To state this succinctly, we will use submatrix notation: Given $\bp \in \dom V$, we denote by $\matr{A}_{\ep}$ the $m \times |\ep|$ the submatrix of $\matr{A}$ with columns indexed by $\ep$, we write $\matr{A}^{\top}_{\ep}$ to denote its transpose, and we denote by $\matr{D}_{\ep}$ the $|\ep| \times |\ep|$ submatrix of signs $\matr{D}(\bp)$ with rows and columns indexed by $\ep$. Finally, we denote by $\bu_{\ep}$ the subvector of $\bu \in \Rn$ indexed by $\ep$. With the notation set, we have the following:

\begin{lem}\label{lem:projection}
Let $t\geqslant 0$ and $\bpo \in \dom \partial_{\bp} V(\cdot;t,\bb)$. The direction of change $\bd(\bpo;t,\bb)$ of the slow system~\eqref{eq:bp_slow_ode} is the unique solution to the cone projection problem
\begin{equation}\label{eq:descent_cone}
\bd(\bpo;t,\bb) = \argmin_{\bd \in \Rm} \normsq{\bd + (\bb + t\bpo)} \ \text{subject to} \ \matr{D}_{\epo}\Atop_{\epo}\bd \geqslant \boldsymbol{0}.
\end{equation}
It admits the characterization
\begin{equation}\label{eq:proj_general}
\bd(\bpo;t,\bb) = \matr{A}_{\epo}\matr{D}_{\epo}\hat{\bu}_{\epo}(\bpo;t,\bb) - (\bb + t\bpo),
\end{equation}
where $\hat{\bu}(\bpo;t,\bb)$ is a solution to the NNLS problem
\begin{equation}\label{eq:nnls_problem}
\hat{\bu}(\bpo;t,\bb) \in \argmin_{\bu \in \Rn} \normsq{\matr{A}\matr{D}(\bpo)\bu - (\bb + t\bpo)} \ \text{subject to} \ 
\begin{cases}
&\bu_{\epo} \geqslant \boldsymbol{0}, \\
&\bu_{\ecpo} = \boldsymbol{0}.
\end{cases}
\end{equation}
Moreover:
\begin{subequations}
\begin{align}
&\hat{\bu}_{j}(\bpo;t,\bb)[\matr{D}(\bpo)\Atop\bd(\bpo;t,\bb)]_{j} = 0 \ \text{for every } j \in \{1,\dots,n\}, \label{eq:nnls_problem_oc_2} \\
&\normsq{\bd(\bpo;t,\bb)} + \langle \bb + t\bpo,\bd(\bpo;t,\bb)\rangle = 0. \label{eq:nnls_problem_oc_3}
\end{align}
\end{subequations}
\end{lem}
\begin{proof} For every $t\geqslant 0$ and $\bpo \in \dom \partial_{\bp} V(\cdot;t,\bb)$, the subdifferential $\partial_{\bp} V(\bpo;t,\bb)$ is non-empty, closed and convex. Hence the projection of $\boldsymbol{0}$ onto $\partial_{\bp} V(\bpo;t,\bb)$ exists and is unique (see Definition~\ref{def:projection}). The projection is given precisely the NNLS problem~\eqref{eq:nnls_problem} or equivalently by its dual problem, the cone projection problem~\eqref{eq:descent_cone}. Equation~\eqref{eq:nnls_problem_oc_2} states the classical Karush-Kuhn-Tucker (KKT) conditions applied to the NNLS problem~\eqref{eq:nnls_problem}. To obtain equation~\eqref{eq:nnls_problem_oc_3}, we write
\[
\begin{alignedat}{1}
\langle\hat{\bu}(\bpo;t,\bb),\matr{D}(\bpo)\Atop\bd(\bpo;t,\bb)\rangle &= \langle\matr{A}\matr{D}(\bpo)\hat{\bu}(\bpo;t,\bb),\bd(\bpo;t,\bb)\rangle \\
&=\langle\matr{A}\matr{D}(\bpo)\hat{\bu}(\bpo;t,\bb) - \bb - t\bp,\bd(\bpo;t,\bb)\rangle \\
&\qquad\qquad + \langle\bb + t\bp,\bd(\bpo;t,\bb)\rangle \\
&= \normsq{\bd(\bpo;t,\bb)} + \langle\bb + t\bp,\bd(\bpo;t,\bb)\rangle,
\end{alignedat}
\]
where the last line follows by~\eqref{eq:proj_general}. Finally, combine~\eqref{eq:proj_general}	with~\eqref{eq:nnls_problem_oc_2} to deduce~\eqref{eq:nnls_problem_oc_3}.
\end{proof}
\begin{rem}\label{rem:computing_nnls}
Problem~\eqref{eq:nnls_problem} may have more than one optimal solution, but the direction $\bd(\bpo;t,\bb)$ is always unique. A review of cone projection algorithms can be found in~\cite{dimiccoli2016fundamentals}. Algorithms for NNLS problems include active set algorithms with ``finite-time" convergence, such as the Lawson--Hanson algorithm~\cite{lawson1995solving} and its generalization, Meyer's algorithm~\cite{meyer2013simple}.
\end{rem}

\section{Dynamics of the slow system}\label{sec:results}

\subsection{The minimal selection is a descent direction}\label{subsec:descent_direction}
Lemma~\ref{lem:projection} shows that the instantaneous direction $\bd(\bp(\tau);t,\bb)$ of the slow system~\eqref{eq:bp_slow_ode} can be calculated from a cone projection problem. The following proposition shows that this characterization implies $\bd(\bp(\tau);t,\bb)$ is a \emph{descent direction} for $V$ and, crucially, obeys an \emph{evolution rule}.
\begin{prop}\label{prop:descent_direction}
Let $t \geqslant 0$ and $\bpo \in \dom \partial_{\bp} V(\cdot;t,\bb)$. Then:
    \begin{itemize}
    \item[(i)] There exists $\Delta_{*}(\bpo;t,\bb) > 0$ such that $\bpo + \Delta\bd(\bpo;t,\bb) \in \dom V(\cdot;t,\bb)$ for every $\Delta \in [0,\Delta_{*}(\bpo;t,\bb)]$, where
    \begin{equation}\label{eq:descent_max_time}
    \Delta_{*}(\bpo;t,\bb) = \min_{j \in \{1,\dots,n\}} \left\{\frac{\sign{\left\langle \matr{D}(\bpo)\Atop\bd(\bpo;t,\bb),\be_j\right\rangle} - \left\langle \matr{D}(\bpo)\Atop\bpo,\be_j\right\rangle)}{\left\langle\matr{D}(\bpo) \Atop\bd(\bpo;t,\bb),\be_j\right\rangle} \right\}.
    \end{equation}
    Moreover, $\Delta_{*}(\bpo;t,\bb) = +\infty \iff \Atop\bd(\bpo;t,\bb) = \boldsymbol{0} \iff \bd(\bpo;t,\bb) = \boldsymbol{0}$.

    \item[(ii)](Descent direction) For every $\Delta \in [0, \Delta_{*}(\bpo;t,\bb)]$, we have
    \begin{equation}\label{eq:V_formula_descent}
    V(\bpo + \Delta\bd(\bpo;t,\bb);t,\bb) - V(\bpo;t,\bb) = -\Delta\left(1-t\Delta/2\right)\normsq{\bd(\bpo;t,\bb)}.
    \end{equation}
    In particular, if $\bd(\bpo;t,\bb) \neq \boldsymbol{0}$, then it is a descent direction of $V(\cdot;t,\bb)$.

    \item[(iii)] For every $\Delta \in [0, \Delta_{*}(\bpo;t,\bb)]$, we have the inclusion
    \[
    -(1-t \Delta)\bd(\bpo;t,\bb) \in \partial_{\bp} V(\bpo + \Delta\bd(\bpo;t,\bb);t,\bb).
    \]

    \item[(iv)] For every $\Delta \in [0,\Delta_{*}(\bpo;t,\bb))$, we have the inclusions
    \[
    \mathcal{E}(\bpo + \Delta\bd(\bpo;t,\bb)) \subset \epo
    \]
    and
    \[
    \partial_{\bp}V(\bpo + \Delta\bd(\bpo;t,\bb);t,\bb) \subset \{t\Delta\bd(\bpo;t,\bb)\} + \partial_{\bp}V(\bpo;t,\bb).
    \]

    \item[(v)] (Evolution rule) For every $\Delta \in [0,\Delta_{*}(\bpo;t,\bb))$, the following evolution rule holds:
    \begin{equation}\label{eq:identity_dd}
    \bd(\bpo + \Delta\bd(\bpo;t,\bb);t,\bb) = (1-t \Delta)\bd(\bpo;t,\bb).
    \end{equation}
\end{itemize}
\end{prop}
\begin{proof}
See Appendix~\ref{app:proof_descent}.
\end{proof}

\subsection{Explicit local solution of the slow system}\label{subsec:local_solution}
The evolution rule~\eqref{eq:identity_dd} describes how the descent direction $\bd(\bpo;t,\bb)$ evolves \emph{locally} in $\Delta > 0$ along the line $\bpo + \Delta\bd(\bpo;t,\bb)$. This evolution is local because it holds on some possibly finite interval $\Delta \in [0,\Delta_{*}(\bpo;t,\bb))$, with $\Delta_{*}(\bpo;t,\bb)$ given by~\eqref{eq:descent_max_time}. Here, we use this evolution rule to show that the slow system evolves as that of its non-projected counterpart, in the sense that 
\[
\dot{\bp}(\tau) = -\proj_{\partial_{\bp} V(\bp(\tau);t,\bb)}(\boldsymbol{0}) \equiv -t(\bp(\tau) - \bpo) + \bd(\bpo;t,\bb) \ \text{ for small enough times $\tau > 0$.}
\]
Its local solution can therefore be computed explicitly, as the following Theorem makes precise:

\begin{thm}\label{thm:mr_local_sol}
Let $t \geqslant 0$ and $\bpo \in \dom \partial_{\bp} V(\cdot;t,\bb)$. The slow system~\eqref{eq:bp_slow_ode} with initial value $\bp(0) = \bpo$ coincides with the initial value problem
\begin{equation}\label{eq:dd_decay}
\dot{\bp}(\tau) = e^{-t\tau}\bd(\bpo;t,\bb), \quad \bp(0) = \bpo
\end{equation}
on $\tau \in [0,\tau_{*}(\bpo;t,\bb))$, where 
\begin{equation}\label{eq:thm_tau0}
\tau_{*}(\bpo;t,\bb) = 
\begin{cases}
    \Delta_{*}(\bpo;0,\bb) & \text{if} \ t = 0, \\
    -\ln\left(1-t\Delta_{*}(\bpo;t,\bb)\right)/t & \text{if} \ t > 0 \ \text{and} \ 1-t\Delta_{*}(\bpo;t,\bb) > 0, \\
    +\infty & \text{otherwise}.
\end{cases}
\end{equation}
In particular, the slow solution is given explicitly on $[0,\tau_{*}(\bpo;t,\bb))$ by
\begin{equation}\label{eq:slow_local_solution}
    \bp(\tau) = \bpo + f(\tau,t)\bd(\bpo;t,\bb),
\end{equation}
where 
\begin{equation}\label{eq:f_fun}
f(\tau,t) = \begin{cases}
    \tau & \text{if} \ t = 0,\\
    \left(1 - e^{-t\tau}\right)/t & \text{if} \ t > 0.
\end{cases}
\end{equation}
\end{thm}

\begin{proof}
We will use the evolution rule~\eqref{eq:identity_dd} to show that the affine system
\begin{equation}\label{eq:affine_system}
    \dot{\bp}_{a}(\tau) =-t(\bpa(\tau) - \bpo) + \bd(\bpo;t,\bb), \qquad \bpa(0) = \bpo.
\end{equation}
satisfies the slow system~\eqref{eq:bp_slow_ode} on $[0,\tau_{*}(\bpo;t,\bb))$ and conclude using uniqueness. First, consider the affine system~\eqref{eq:affine_system}. A short calculation shows that its unique, global solution is given by 
\[
\bpa(\tau) = \bpo + f(\tau,t)\bd(\bpo;t,\bb).
\]
Substitute the solution in~\eqref{eq:affine_system} above to find
\begin{equation}\label{eq:explicit_slow_evolution}
\begin{alignedat}{1}
\dot{\bp}_{a}(\tau) &= (1-tf(\tau,t))\bd(\bpo;t,\bb) \\
&= e^{-t\tau}\bd(\bpo;t,\bb).
\end{alignedat}
\end{equation}
Next, we invoke the evolution rule~\eqref{eq:identity_dd} in Proposition~\ref{prop:descent_direction}(v) with $\Delta = f(\tau,t)$ to find
\begin{equation}\label{eq:thm_local_id_holds}
-\proj_{\partial_{\bp} V(\bpo + f(\tau,t)\bd(\bpo;t,\bb),t)}(\boldsymbol{0}) = (1-tf(\tau,t))\bd(\bpo;t,\bb)
\end{equation}
whenever $0 \leqslant f(\tau,t) < \Delta_{*}(\bpo;t,\bb)$. Notice $\tau \mapsto f(\tau,t)$ increases monotonically from $0$ to $1/t$ (0 to $+\infty$) when $t > 0$ ($t = 0$). Hence if $t = 0$, the largest value of $\tau$ for which~\eqref{eq:thm_local_id_holds} holds is $\Delta_{*}(\bpo;0,\bb)$. If $t > 0$ and $1-t\Delta_{*}(\bpo;t,\bb) > 0$, then $(1-e^{-t\tau})/t$ is equal to $\Delta_{*}(\bpo;t,\bb)$ at $\tau = -\ln\left(1-t\Delta_{*}(\bpo;t,\bb)\right)/t$. If $t > 0$ and $1 - t\Delta_{*}(\bpo;t,\bb)) \leqslant 0$, then~\eqref{eq:thm_local_id_holds} holds for every $\tau \geqslant 0$. Taken together, we find 
\[
\dot{\bp}_{a}(\tau) = -\proj_{\partial_{\bp} V(\bpo + f(\tau,t)\bd(\bpo;t,\bb),t)}(\boldsymbol{0}) \ \text{for every $\tau \in [0,\tau_{*}(\bpo;t,\bb))$}.
\]
Uniqueness follows from Theorem~\ref{thm:eu_diff_inclusions}, and hence~\eqref{eq:dd_decay} follows from~\eqref{eq:explicit_slow_evolution}.
\end{proof}

\subsection{Continuation of the local solution}\label{subsec:continuation}
Theorem~\ref{thm:mr_local_sol} provides the explicit local solution of the slow system~\eqref{eq:bp_slow_ode} on the interval $[0,\tau_{*}(\bpo;t,\bb))$. When $\tau_{*}(\bpo;t,\bb) < +\infty$, what can we say about the slow solution and system at $\tau = \tau_{*}(\bpo;t,\bb)$?
\begin{lem}\label{lem:discontinuity}
Let $t \geqslant 0$, $\bpo \in \dom \partial_{\bp} V(\cdot;t,\bb)$, and suppose $\tau_{*}(\bpo;t,\bb) < +\infty$. Then~\eqref{eq:slow_local_solution} holds at $\tau = \tau_{*}(\bpo;t,\bb)$ but~\eqref{eq:dd_decay} does not:
\begin{equation}\label{eq:dd_discontinuity}
\normtwo{\bd(\bp(\tau_{*}(\bpo;t,\bb));t,\bb)} < e^{-t\tau_{*}(\bpo;t,\bb)}\normtwo{\bd(\bpo;t,\bb)}.
\end{equation}
\end{lem}
\begin{proof}
Suppose $\tau_{*}(\bpo;t,\bb) < +\infty$. By Proposition~\ref{prop:descent_direction}(iii) and Theorem~\ref{thm:mr_local_sol}, we have
\[
\frac{d\bp}{d\tau}(\tau_{*}(\bpo;t,\bb)), \frac{d\bpa}{d\tau}(\tau_{*}(\bpo;t,\bb)) \in -\partial_{\bp} V(\bp(\tau_{*}(\bpo;t,\bb));t,\bb).
\]
Hence we can extend the analysis done in Theorem~\ref{thm:mr_local_sol} to the value $\tau = \tau_{*}(\bpo;t,\bb)$, yielding $\bp(\tau_{*}(\bpo;t,\bb)) = \bpa(\tau_{*}(\bpo;t,\bb))$. Thus~\eqref{eq:slow_local_solution} holds at $\tau = \tau_{*}(\bpo;t,\bb)$.

Next, we prove inequality~\eqref{eq:dd_discontinuity}. Note that Proposition~\ref{prop:descent_direction}(iii) immediately implies
\[
\normtwo{\bd(\bp(\tau_{*}(\bpo;t,\bb));t,\bb)} \leqslant e^{-t\tau_{*}(\bpo;t,\bb)}\normtwo{\bd(\bpo;t,\bb)}.
\]
In addition, since the projection onto the closed and convex set $\partial_{\bp} V(\bp(\tau_{*}(\bpo;t,\bb));t,\bb)$ is unique (\ref{def:projection}), it sufficies to show that $\bd(\bp(\tau_{*}(\bpo;t,\bb));t,\bb) \neq e^{-t\tau_{*}(\bpo;t,\bb)}\bd(\bpo;t,\bb)$. Now, apply Theorem~\ref{thm:mr_local_sol} with the initial value $\bp(\tau_{*}(\bpo;t,\bb))$ at $\tau = \tau_{*}(\bpo;t,\bb)$ to find
\begin{equation}\label{eq:prop_discontinuity}
\bp(\tau) = \bp(\tau_{*}(\bpo;t,\bb)) + f(\tau-\tau_{*}(\bpo;t,\bb),t)\bd(\bp(\tau_{*}(\bpo;t,\bb));t,\bb)
\end{equation}
on the interval $[\tau_{*}(\bpo;t,\bb),\tau_{*}(\bp(\tau_{*}(\bpo;t,\bb));t,\bb)]$. Suppose, for a contradiction, that
\[
\bd(\bp(\tau_{*}(\bpo;t,\bb));t,\bb) = e^{-t\tau_{*}(\bpo;t,\bb)}\bd(\bpo;t,\bb).
\] 
Then~\eqref{eq:prop_discontinuity} can be developed into
\[
\bp(\tau) = \bp_{0} + \left(\Delta_{*}(\bpo;t,\bb) + f(\tau-\tau_{*}(\bpo;t,\bb),t)e^{-t\tau_{*}(\bpo;t,\bb)}\right)\bd(\bpo;t,\bb).
\]
However, this implies $\Delta_{*}(\bpo;t,\bb) + f(\tau-\tau_{*}(\bpo;t,\bb),t)e^{-t\tau_{*}(\bpo;t,\bb)} > \Delta_{*}(\bpo;t,\bb)$ for every $\tau \in (\tau_{*}(\bpo;t,\bb),\tau_{*}(\bp(\tau_{*}(\bpo;t,\bb));t,\bb)]$, contradicting the definition of $\Delta_{*}(\bpo;t,\bb)$ in~\eqref{eq:descent_max_time}. We therefore conclude $\bd(\bp(\tau_{*}(\bpo;t,\bb));t,\bb) \neq e^{-t\tau_{*}(\bpo;t,\bb)}\bd(\bpo;t,\bb)$.
\end{proof}

\subsection{Explicit global solution of the slow system}\label{subsec:global}
So far, we have computed the explicit local solution of the slow system~\eqref{eq:bp_slow_ode} on an interval $[0,\tau_{*}(\bpo;t,\bb))$, which can be extended to $\tau = \tau_{*}(\bpo;t,\bb)$ when the value is finite. We can apply Theorem~\ref{thm:mr_local_sol} and Lemma~\ref{lem:discontinuity} again with the initial condition $\bp(\tau_{*}(\bpo;t,\bb))$ to extend the local solution to $[0, \tau_{*}(\bp(\tau_{*}(\bpo;t,\bb));t,\bb))$. This is because $\bp(\tau_{*}(\bpo;t,\bb)) \in \dom \partial_{\bp}V(\cdot;t,\bb)$ and all assumptions of Theorem~\ref{thm:mr_local_sol} hold.

We now apply this argument repeatedly to compute explicitly the global solution of the slow system. Let $\bp^{(0)} \in \dom \partial_{\bp}V(\cdot;t,\bb)$, $\tau^{(0)} = 0$, and starting from $k=1$ let
\begin{equation}\label{eq:iterations_global_sol}
\begin{alignedat}{1}
	\bd^{(k-1)} &= \bd(\bp^{(k-1)};t,\bb),\\
	\Delta^{(k-1)} &= \Delta_{*}(\bp^{(k-1)};t,\bb),\\
	\tau^{(k)} &= \tau_{*}(\bp^{(k-1)};t,\bb),\\
	\bp^{(k)} &= \bp^{(k-1)} 
+ \begin{cases}
 	\min\left(\Delta^{(k-1)},1/t\right)\bd^{(k-1)} & \text{if} \ t>0 \\
 	\Delta^{(k-1)}\bd^{(k-1)} & \text{if $t = 0$ and $\Delta^{(k-1)} < +\infty$}, \\
 	\boldsymbol{0} & \text{otherwise.}
 \end{cases}
\end{alignedat}
\end{equation} 
The slow solution is given piecewise on the intervals $[\tau^{(k-1)},\tau^{(k)})$:
\begin{equation}\label{eq:slow_global_solution}
\bp(\tau) = \bp^{(k-1)} + f(\tau - \tau^{(k-1)},t)\bd^{(k-1)} \ \text{over} \ \tau \in [\tau^{(k-1)},\tau^{(k)}).
\end{equation}
This gives the explicit solution of the slow system over $[0,\tau^{(k)})$ up to any $k \in \mathbb{N}$. Note that if $t>0$, then the minimum in~\eqref{eq:iterations_global_sol} is attained at $1/t \iff \tau^{(K)} = +\infty$ for some $K \in \mathbb{N}$.

What can be said about the asymptotic limit $k\to+\infty$? The following Theorem asserts the limit converges in \emph{finite time}, in the sense that there exists some nonnegative integer $K$ such that $\bp(\tau) = \bp^{(K)}$ over the interval $[\tau^{(K)},+\infty)$.

\begin{thm}\label{thm:finite_time}
Let $t \geqslant 0$ and $\bp^{(0)} \in \dom \partial_{\bp}V(\cdot;t,\bb)$. Consider the slow system~\eqref{eq:bp_slow_ode} with initial condition $\bp(0) = \bp^{(0)}$, whose solution is given by~\eqref{eq:slow_global_solution}. Then there exists a nonnegative integer $K$ such that on the interval $\tau \in [\tau^{(K)},+\infty)$,
\[
\bp(\tau) = \begin{cases}
\bp^{(K)} + f(t,\tau-\tau^{(K)})\bd^{(K)} &\, \text{if $t>0$}, \\
\bp^{(K)} &\,\text{if $t=0$}.
\
\end{cases}
\]
\end{thm}
\begin{proof}
We will show there exists a nonnegative integer $K$ such that $t\Delta^{(K)} \geqslant 1$ when $t>0$ or $\Atop\bd^{(K)} = \boldsymbol{0}$ when $t = 0$. Let $k$ be a positive integer and suppose $\tau^{(j)} < +\infty$ for every $j \in \{1,\dots,k\}$ in~\eqref{eq:iterations_global_sol}. First, note that the directions $\{\bd^{(1)},\dots,\bd^{(k)}\}$ are obtained from projections landing on different faces of the convex cone $\partial_{\bp} V(\cdot;t,\bb)$. Indeed, this follows because Lemma~\ref{lem:discontinuity} implies $\normtwo{\bd^{(j)}} < \normtwo{\bd^{(j-1)}}$ for $j \in \{1,\dots,k\}$ and because the projection onto the closed, convex set $\partial_{\bp} V(\cdot;t,\bb)$ is unique (see~\ref{def:projection}). Since $\rank{\matr{A}} = m$ by assumption, there is at least one face on which the norm of the projection onto $\partial_{\bp} V(\cdot;t,\bb)$ is equal to zero. Thus there exists some $K > k$ such that either $\bd^{(K)} = \boldsymbol{0}$ or, when $t>0$, $t\Delta^{(K)} \geqslant 1$, whichever happens first (recall from Proposition~\ref{prop:descent_direction}(i) that $\bd^{(K)} = \boldsymbol{0}$ if and only if $\Delta^{(K)} = +\infty$). This yields the desired result.
\end{proof}

\begin{rem}\label{rem:finite_time_holds}
The proof above shows that finite-time convergence holds even when $\bb \notin \text{Im}(\matr{A})$. However, in that case there may be some $K \in \mathbb{N}$ for which $\bd^{(K-1)} \neq 0$ with $\Atop\bd^{(K-1)} = \boldsymbol{0}$ (this will always be the case when $t=0$). This causes no issues when $t>0$ because $\bp^{(K)}$ remains finite. However, when $t=0$ we have $\lim_{\tau \to +\infty}\normtwo{\bp(\tau)} = +\infty$, which means the slow solution diverges and there are no feasible solutions to the corresponding~\eqref{eq:BP_primal} and~\eqref{eq:BP_dual} problems.
\end{rem} 

\subsection{An exact algorithm for recovering optimal solutions to the lasso problem}\label{subsection:recovery_sols}
The analysis in Sections~\ref{subsec:local_solution}--\ref{subsec:global} yields the global solution to the slow system~\eqref{eq:bp_slow_ode}. Crucially, its asymptotic limit can be computed explicitly from~\eqref{eq:slow_global_solution} because it converges in finite time, meaning $\lim_{\tau \to +\infty}\bp(\tau) = \bp^{(K)}$ for some nonnegative integer $K$. This recovers a pair of primal and dual lasso solutions $(\bx^{s}(\tzero,\bb),\bp^{s}(t,\bb))$. Indeed, the minimal selection principle implies $\bp^{s}(\tzero,\bb) = \bp^{(K)}$, while a primal solution follows from the optimality conditions~\eqref{eq:bpdn_full_oc} and Lemma~\ref{lem:projection}: Letting $\hat{\bu}^{(K)}$ denote an optimal solution to the NNLS problem~\eqref{eq:nnls_problem}, then
\[
-t\bp^{(K)} = \bb - \matr{A}\matr{D}(\bp^{(K)})\hat{\bu}^{(K)} \implies \bx^{s}(t,\bb) = \matr{D}(\bp^{(K)})\hat{\bu}^{(K)}.
\]

Our results yield Algorithm~\ref{alg:lasso}, presented in the introduction of this paper. The algorithm integrates the slow system~\eqref{eq:bp_slow_ode} and computes its asymptotic limit in a finite number of steps. The slow solution is calculated \emph{exactly} because the descent directions and timesteps can be computed using, e.g., active set NNLS or cone projection algorithms as per Remark~\ref{rem:computing_nnls}.

Note that Algorithm~\ref{alg:lasso} requires an input $\bp^{(0)} \in \Rm$ with $\norminf{\Atop\bp^{(0)}} \leqslant 1$. For this, one can always take $\bp^{(0)} = -\bb/\norminf{\Atop\bb}$. In addition, Algorithm~\ref{alg:lasso} is particularly well suited for computing regularization paths. To do so, one selects a decreasing sequence of nonnegative hyperparameters $\{t^{(l)}\}_{l=0}^{L}$ (starting perhaps from $\bp^{s}(t^{(0)},\bb) = -\bb/t^{(0)}$) and sequentially computes the dual solutions $\{\bp^{s}(t^{(l)},\bb)\}_{l=0}^{L}$, using $\bp^{s}(t^{(l)},\bb)$ as the input in Algorithm~\ref{alg:lasso} for computing $\bp^{s}(t^{(l+1)},\bb)$.
\section{Continuation in the hyperparameter of the asymptotic limit}\label{sec:perturbations}
Sections~\ref{sec:msp}--\ref{sec:results} provide a novel characterization of the lasso problem using the dynamics of the slow system~\eqref{eq:bp_slow_ode} and its slow solution~\eqref{eq:slow_global_solution} for \emph{fixed} hyperparameter $t$. In practice, one often seeks to compute the solution path $t \mapsto (\bx^{s}(t,\bb),\bp^{s}(t,\bb))$ to, e.g., assess robustness of solutions~\cite{donoho2006stability,fosson2020sparse} or select a hyperparameter using data-driven methods~\cite{friedman2010regularization,osborne2000new,efron2004least,donoho2008fast,bringmann2018homotopy}. This historically motivated the development of the LARS algorithm~\cite{osborne2000new,efron2004least}, which is now well-known to fail without technical assumptions that are difficult to verify~\cite[Proposition 4.1]{bringmann2018homotopy}.

Here, we present a rigorous analysis of continuation of the asymptotic limit of the slow system~\eqref{eq:bp_slow_ode} in terms of the hyperparameter $t$, yielding lasso solution paths. We present in Section~\ref{subsec:generalized_homotopy} the local dependence of the slow solution on $t$. As we argue in Section~\ref{subsec:non-uniqueness}, this local dependence leads to a possibly non-unique local continuation of the primal lasso solution. This non-uniqueness issue, a well-known problem in the literature~\cite{tibshirani2013lasso}, arises from the non-uniqueness of solutions to an NNLS problem, as previously reported by~\cite{bringmann2018homotopy}. Finally, we present in Section~\ref{subsec:generalized_homotopyII} the global dependence of the slow solution, on $t$, naturally yielding a rigorous homotopy algorithm based on the minimal selection principle.


\subsection{Local dependence on hyperparameter of the slow solution}\label{subsec:generalized_homotopy}
We first describe how the descent direction $\bd(\bpo;\tzero,\bb)$ changes under perturbations in the hyperparameter and data. 
\begin{lem}\label{lem:tech_perturbation_I}
Let $\tzero \geqslant 0$, let $\bpo \in \dom V(\cdot;\tzero,\bb)$, and let $\hat{\bu}(\bpo;\tzero,\bb)$ be a global minimum of the NNLS problem~\eqref{eq:nnls_problem} so that $\bd(\bpo;\tzero,\bb) = \matr{A}\matr{D}(\bpo)\hat{\bu}(\bpo;\tzero,\bb) -(\bb +\tzero\bpo)$. In addition, let 
	$\delta_{0} \in \R$ be such that $\tzero + \delta_{0} \geqslant 0$. Then: 
	
	\begin{itemize}
	\item[(i)] The perturbed descent direction $\bd(\bpo;\tzero + \delta_{0},\bb)$ is given by 
	\begin{equation}\label{eq:descent_direction_change_in_t}
            \bd(\bpo;\tzero + \delta_{0},\bb) = \bd(\bpo;\tzero, \bb) + \matr{A}\matr{D}(\bpo)\hat{\bv}(\bpo;t,\bb,\delta_{0},\bdelta\bb) - \delta_{0} \bpo,
	\end{equation}
    where
    \begin{equation}\label{eq:nnls_problem_v}
    \begin{alignedat}{1}
        &\hat{\bv}(\bpo;\tzero,\bb,\delta_{0}) \in \argmin_{\bv \in \Rn} \normsq{\bd(\bpo;\tzero, \bb) + \matr{A}\matr{D}(\bpo)\bv - \delta_{0} \bpo} \\
        &\qquad\qquad \text{subject to} \
            \begin{cases}
                \bv_{\epo} &\geqslant -\hat{\bu}_{\epo}(\bpo;\tzero,\bb), \\
                \bv_{\ecpo} &= \boldsymbol{0}.
            \end{cases}
        \end{alignedat}
	\end{equation}

	\item[(ii)] (Linear perturbations in hyperparameter) Suppose $\bd(\bpo;\tzero,\bb) = \boldsymbol{0}$ and let $\delta_{0} = t-\tzero$ with $t\in[0,\tzero]$. Then~\eqref{eq:descent_direction_change_in_t} and~\eqref{eq:nnls_problem_v} refine to
	\begin{equation}\label{eq:descent_direction_change_in_t_2}
            \bd(\bpo;t,\bb) =  (1-t/\tzero) \left(\matr{A}\matr{D}(\bpo)\hat{\bv}(\bpo;\tzero, t,\bb) + t_{0}\bpo\right)
	\end{equation}
    where
    \begin{equation}\label{eq:nnls_problem_v_2}
    \begin{alignedat}{1}
        &\hat{\bv}(\bpo;\tzero, t,\bb) \in \argmin_{\bv \in \Rn} \normsq{\matr{A}\matr{D}(\bpo)\bv + t_{0}(\bpo)} \\
        &\text{subject to} \
            \begin{cases}
                &\bv_{j} \geqslant -\hat{\bu}_{j}(\bpo;\tzero,\bb)/(1-t/\tzero) \ \text{if $j \in \epo$}, \\
                &\bv_{\ecpo} = \boldsymbol{0}.
            \end{cases}
        \end{alignedat}
	\end{equation}	
\end{itemize}
\end{lem}
\begin{proof}
See Appendix~\ref{app:tech_perturbation_I} for the proof.	
\end{proof}

In the stationary case $\bd(\bpo;\tzero,\bb) = \boldsymbol{0}$ with $\tzero>0$, Lemma~\ref{lem:tech_perturbation_I} characterizes \emph{implicitly} how the primal and dual lasso solutions change under perturbations in the hyperparameter. Next, we show how these perturbations characterize these changes \emph{explicitly}.
\begin{prop}\label{prop:generalized_homotopy}
Let $\tzero > 0$ and $(\bx^{s}(\tzero,\bb),\bp^{s}(\tzero,\bb))$ denote a pair of primal and dual lasso solutions at hyperparameter $\tzero$ and data $\bb$. In addition, let
	\begin{equation}\label{eq:nnls_problem_v_3}
    \begin{alignedat}{1}
        &\hat{\bv}^{s}(\tzero,\bb) \in \argmin_{\bv \in \Rn} \normsq{\matr{A}\matr{D}(\bp^{s}(\tzero,\bb))\bv + t_{0}(\bp^{s}(\tzero,\bb))} \\
        &\text{subject to} \
            \begin{cases}
                \bv_{j} &\geqslant 0 \ \text{if $j \in \mathcal{E}(\bp^{s}(\tzero,\bb))$ and $\bx^{s}_{j}(\tzero,\bb) = 0$}, \\
                \bv_{j} &= 0 \ \text{if $j \in \mathcal{E}^{\mathsf{C}}(\bp^{s}(\tzero,\bb))$},
            \end{cases}
        \end{alignedat}
	\end{equation}	
	and define
	\begin{equation}\label{eq:homotopy_parameters}
    \begin{alignedat}{1}
    &\bxi^{s}(\tzero,\bb) \coloneqq \matr{A}\matr{D}(\bp^{s}(\tzero,\bb))\hat{\bv}^{s}(\tzero,\bb) + \tzero(\bp^{s}(\tzero,\bb)), \\
    & C^{s}(\tzero, \bb) \coloneqq \\
    &\quad \inf_{j \in \{1,\dots,n\}} \left\{\frac{\sign{\left(\langle \matr{D}(\bp^{s}(\tzero,\bb))\Atop\bxi^{s}(\tzero,\bb),\be_j\rangle\right)} - \langle \matr{D}(\bp^{s}(\tzero,\bb))\Atop\bp^{s}(\tzero,\bb),\be_j\rangle)}{\langle\matr{D}(\bp^{s}(\tzero,\bb)) \Atop\bxi^{s}(\tzero,\bb),\be_j\rangle} \right\}, \\
    & T_{+}(\tzero,\bb) \coloneqq \frac{\tzero}{1 + \tzero C^{s}(\tzero, \bb)}, \qquad T_{-}(\tzero,\bb,\hat{\bv}^{s}) \coloneqq  \tzero \left(1 - \inf_{\substack{j \in \eqset(\bp^{s}(\tzero,\bb)) \\ \bx^{s}_{j}(\tzero,\bb) \neq 0 \\ \hat{\bv}_{j}^{s}(\tzero,\bb) \leqslant -|\bx_{j}^{s}(\tzero,\bb)|}} \frac{|\bx_{j}^{s}(\tzero,\bb)|}{|\hat{\bv}^{s}_{j}(\tzero,\bb)|}\right), \\
    & t_{1} \coloneqq \max\left(T_{-}(\tzero,\bb,\hat{\bv}^{s}),T_{+}(\tzero,\bb)\right).
    \end{alignedat}
    \end{equation} 
    Then $0 \leqslant t_{1} < \tzero$ and, for every $t \in [t_{1}, \tzero]$,
    \begin{equation}\label{eq:homotopy_solutions}
    \begin{alignedat}{1}
    	\bx^{s}(t,\bb) &= \bx^{s}(\tzero,\bb) + \left(1 - \frac{t}{\tzero}\right) \matr{D}(\bp^{s}(\tzero,\bb))\hat{\bv}^{s}(\tzero,\bb), \\
    	\bp^{s}(t,\bb) &= \begin{cases}
 		&\bp^{s}(\tzero,\bb) + \left(\frac{1}{t} - \frac{1}{\tzero}\right)\bxi^{s}(\tzero,\bb) \quad \text{if $t_{1}(\tzero,\bb) > 0$}, \\
 		&\bp^{s}(\tzero,\bb) \quad \text{otherwise.}
 		\end{cases}
    \end{alignedat}
    \end{equation}
\end{prop}
\begin{proof}
See Appendix~\ref{app:generalized_homotopy} for the lengthy and technical proof.
\end{proof}

\begin{rem}\label{rem:kinks}
The numbers $T_{+}(\tzero,\bb)$ and $T_{-}(\tzero,\bb,\hat{\bv}^{s})$ identify potential ``kinks" in the piecewise linear dependence of the primal and dual solutions on $t$. If $t_{1} = T_{+}(\tzero,\bb) > 0$, then there is at least one index $j \in \{\{1,\dots,n\} : \bx^{s}_{j}(\tzero,\bb) = 0\}$ that joins the set $\eqset(\bp^{s}(t_{1},\bb))$, and this index is new if $j \in \eqcset(\bp^{s}(\tzero,\bb))$. If $t_{1} = T_{-}(\tzero,\bb,\hat{\bv}^{s}) > 0$, then there is at least one index $j \in \eqset(\bp^{s}(\tzero,\bb))$ such that $\bx^{s}_{j}(t_{1}, \bb) = \boldsymbol{0}$. This index may leave the equicorrelation set $\eqset(\bp^{s}(t_{1},\bb-(\tzero-t_{1})))$ beyond $t > t_{1}$.
\end{rem}


\subsection{Non-uniqueness in the local dependence on hyperparameter}\label{subsec:non-uniqueness}
The local continuation of $(\bx^{s}(\tzero,\bb),\bp^{s}(\tzero,\bb))$ depends on the solution $\bv^{s}(\tzero,\bb)$ and residual vector $\bxi^{s}(\tzero,\bb)$ of the NNLS problem~\eqref{eq:nnls_problem_v_3}. Since the residual vector is unique, the number $T_{+}(\tzero,\bb)$ and the continuation of $\bp^{s}(\tzero,\bb)$ are also unique. The NNLS solution, however, is generally not unique. This leads to complications because the number $T_{-}(\tzero,\bb,\hat{\bv}^{s})$ may not be unique. Thus the continuation of $\bx^{s}(t,\bb)$ beyond $t < \tzero$ may not be unique. This is not new and is a well-known problem~\cite{tibshirani2013lasso,bringmann2018homotopy}.

To understand how the non-uniqueness of the NNLS problem~\eqref{eq:nnls_problem_v_3} arises, it helps to identify when the solution is unique. If the least-squares solution 
\[
\bvlsq = \tzero(\Atop_{\eqset(\bp^{s}(\tzero,\bb)}\matr{A}_{\eqset(\bp^{s}(\tzero,\bb))})^{-1} \matr{D}(\bp^{s}(\tzero,\bb))\Atop_{\eqset(\bp^{s}(\tzero,\bb))}(\bp^{s}(\tzero,\bb))
\]
satisfies $\bvlsq_{j} \geqslant 0$ for every $j \in \eqset(\bp^{s}(\tzero,\bb))$ with $\bx_{j}^{s}(\tzero,\bb) = 0$ and $\matr{A}_{\eqset(\bp^{s}(\tzero,\bb))}$ has full column rank, then it is the unique solution. If instead $\matr{A}_{\eqset(\bp^{s}(\tzero,\bb))}$ has full row rank, then $\bvlsq$ is the unique solution with least $\ell_{2}$ norm. Thus one possibility is to select the solution with least $\ell_{2}$-norm. However, in practice we have found that it was unnecessary; see Remark~\ref{rem:homotopy_convergence}.


\subsection{Global dependence on hyperparameter of the slow solution}\label{subsec:generalized_homotopyII}
From a dynamical systems perspective, Lemma~\ref{lem:tech_perturbation_I} and Proposition~\ref{prop:generalized_homotopy} describe the local continuation of the asymptotic limit of the slow system~\eqref{eq:bp_slow_ode} in terms of the hyperparameter. This continuation is local because it holds on an interval $t \in [t_{1},\tzero]$, and it is not unique whenever the solution to the NNLS problem~\eqref{eq:nnls_problem_v_3} is not unique. Nevertheless, if $t_{1} \neq 0$, we can apply Proposition~\ref{prop:generalized_homotopy} again to extend the local continuation to an interval $[t_{2},\tzero]$ with $0 \leqslant t_{2} < t_{1}$.

We now apply Proposition~\ref{prop:generalized_homotopy} repeatedly, starting from $t^{(0)} = \norminf{\Atop\bb}$, $\bx^{(0)} = \boldsymbol{0}$, and $\bp^{(0)} = -\bb/\norminf{\Atop\bb}$. Doing so $k$ times yields a global continuation of the asymptotic limit of the slow system~\eqref{eq:bp_slow_ode}, hence a global continuation of the pair of primal and dual solutions $(\bx^{s}(t,\bb),\bp^{s}(t,\bb))$ on $t \in [t^{(k)},t^{(0)}]$. As discussed in Section~\ref{subsec:non-uniqueness}, the continuation in the primal solution $\bx^{s}(t,\bb)$ is not unique. Choosing at each $k$ the NNLS solution to problem~\eqref{eq:nnls_problem_v_3} with least $\ell_{2}$ norm, we obtain Algorithm~2 below. It is nearly identical to the generalized homotopy algorithm proposed by Bringmann et al.~\cite{bringmann2018homotopy}, with the exception of the breakpoints, which here are simpler and more explicit. Algorithm~2 yields a lasso solution path and converges in finite time.
\begin{algorithm}[ht]
	\label{alg:homotopy_alg}
    \caption{Homotopy algorithm for computing the primal and dual solution paths to the lasso problem.}
    \Input{A matrix $\matr{A} \in \mathbb{R}^{m \times n}$ and a vector $\bb \in \image{\matr{A}}\setminus{\{\boldsymbol{0}\}}$.}
    \Output{A finite sequence $\left\{t^{(k)}, \bx^{(k)},\bp^{(k)}\right\}_{k=0}^{K}$ specifying a solution path to the lasso problem.}
    \BlankLine
    \emph{Set $t^{(0)} = \norminf{\Atop\bb}$, $\bx^{(0)} = \boldsymbol{0}$ and $\bp^{(0)} = -\bb/\norminf{\Atop\bb}$}\;
    \For{$k = 1$ until convergence} {
    	Compute $\eqset^{(k-1)} = \{j \in \{1,\dots,n\} : |\langle-\Atop\bp,\be_{j}\rangle| = 1\}$\;
    	Compute $\matr{D}^{(k-1)} = \diag\left(\sign{(-\Atop\bp^{(k-1)}})\right)$\;
        Compute the solution $\hat{\bv}^{(k-1)} \in \argmin_{\bv \in \Rn} \normsq{\matr{A}_{\eqset^{(k-1)}}\matr{D}_{\eqset^{(k-1)}}^{(k-1)}\bv + t^{(k-1)}\bp^{(k-1)}}$  subject to \[
        \begin{cases}
 	\bv_{j} \geqslant 0,\, &\text{if $j \in \eqset^{(k-1)}$ and $\bx_{j}^{(k-1)} = 0$} \\
 	\bv_{j} = 0, \, &\text{if $j \in \left(\eqset^{(k-1)}\right)^{\mathsf{C}}$}
 \end{cases}
 \]
 with minimal $\ell_{2}$ norm\;
 Compute $\bxi^{(k-1)} = \matr{A}\matr{D}(\bp^{(k-1)})\hat{\bv}^{(k-1)} + t^{(k-1)}\bp^{(k-1)}$ \;
        Compute the numbers $T_{-}(t^{(k-1)},\bb,\hat{\bv}^{(k-1)})$ and $T_{+}(t^{(k-1)},\bb)$ from~\eqref{eq:homotopy_parameters}\;
        Update $t^{(k)} = \max(T_{-}(t^{(k-1)},\bb,\hat{\bv}^{(k-1)}),T_{+}(t^{(k-1)},\bb))$\;
        Update $\bx^{(k)} = \bx^{(k-1)} + \left(1-t^{(k)}/t^{(k-1)}\right)\matr{D}(\bp^{(k-1)})\hat{\bv}^{(k-1)}$\; 
        \uIf{$t^{(k)} = 0$} {
        	$\bp^{(k)} = \bp^{(k-1)}$ \;
        	break \tcp*[h]{The algorithm has converged}\;
        }
        \uElse {
        Update $\bp^{(k)} = \bp^{(k-1)} + (1/t^{(k)} - 1/t^{(k-1)})\bxi^{(k-1)}$\;
        }
    }
\end{algorithm}
\begin{thm}\label{thm:classical_homotopy}
The global continuation method described in Algorithm~2 is correct and converges in finite time, that is, there exists a nonnegative integer $K$ such that $t^{(K)} = 0$.
\end{thm}
\begin{proof}
Correctness follows from Proposition~\ref{prop:generalized_homotopy} and the discussion in Section~\ref{subsec:non-uniqueness}. Convergence in finite time is identical to that of Theorem~4.2 in Bringmann et al.~\cite{bringmann2018homotopy} and is omitted.
\end{proof}
\begin{rem}\label{rem:homotopy_convergence}
In practice, we have found that Algorithm~2 always converged when using the Lawson--Hanson algorithm~\cite{lawson1995solving} or Meyers's algorithm~\cite{meyer2013simple} to compute a solution to the NNLS problem on line 5 without explicitly finding the minimal $\ell_{2}$-norm solution, including on pathological examples such at those in, e.g., ~\cite{mairal2012complexity,bringmann2018homotopy}.
\end{rem}
\section{Numerical experiments}\label{sec:numerics}
This section presents some numerical experiments to compare the accuracy and run times of \textbf{Algorithm~1} with some state-of-the-art algorithms for computing regularization and solution paths to~\eqref{eq:lasso_primal}. Specifically, we compare our algorithms with the glmnet software package \textbf{glmnet}~\cite{tibshirani1996regression,zou2005regularization,friedman2010regularization,hastie2021glmnet}, MATLAB's native lasso implementation \textbf{mlasso} (ostensibly the implementation of the glmnet package to MATLAB, except that it has more options, including for better control of the accuracy), and the fast iterative shrinkage thresholding algorithm \textbf{fista} with a strong selection rule~\cite{beck2009fast,ghaoui2010safe,tibshirani2012strong,raj2016screening}. For the special case $t=0$, we compare the performance of \textbf{Algorithm~1} with regularization path starting from $(t^{(0)},\bp^{(0)}) = (\norminf{\Atop\bb}, -\bb/t^{(0)})$ and no regularization path starting from input $\bp^{(0)} = -\bb/\norminf{\Atop\bb}$. All experiments were carried out in MATLAB on an Apple M3 processor.

\paragraph{Implementations and datasets.} All our algorithms require solving a NNLS problem. For this, we use Meyer's algorithm~\cite{meyer2013simple}, an active set NNLS algorithm generalizing the Lawson--Hanson algorithm~\cite{lawson1995solving} to arbitrary starting active sets, and so particularly suitable for Algorithm~1. For \textbf{glmnet}, we use the code available at~\url{https://github.com/junyangq/glmnet-matlab}. For \textbf{fista}, we use the variant described in~\cite[Algorithm 5, Page 197]{chambolle2016introduction} together with the selection rule described in~\cite[Equation 7]{tibshirani2012strong}. For our numerical experiments, we use the benchmark datasets provided by Lorenz et al.~\cite[Tables 1-2, Section 4]{lorenz2015solving}, tailored specifically for~\eqref{eq:lasso_primal} and~\eqref{eq:BP_primal} and to use for comparing different numerical algorithms. Each dataset comprises a triplet $(\matr{A},\bb,\bx_{\text{BP}}^{s})$, where $\bx_{\text{BP}}^{s} = \argmin_{\bx \in \Rm}\normone{\bx} \ \text{s.t.} \ \matr{A}\bx=\bb$. We use in total ten different datasets, corresponding to some of their largest dense matrix ($m = 1024,\, n = 8192$) with six different observed data and solution vectors and their largest sparse matrix ($m = 8192,\, n = 49152$) with four different observed data and solution vectors. The solution vectors have either high dynamic range (HDR) or low dynamic range (LDR), and their support either satisfy the so-called Exact Recovery Condition (ERC)~\cite{tropp2006just}, an extended form of it (extERC) or, for dense matrices only, no exact recovery condition (noERC) (see~\cite{lorenz2015solving} for details). All our MATLAB implementations, external software and datasets will be bundled and made publicly available on Zotero.

\subsection{Accuracy checks}\label{subsec:numerics_accuracy}
We compare how accurately \textbf{glmnet}, $\textbf{fista}$ and $\textbf{mlasso}$ compute solutions to~\eqref{eq:lasso_primal} and~\eqref{eq:lasso_dual} (via the optimality conditions $t\bp^{s}(t,\bb) = \matr{A}\bx^{s}(t,\bb)-\bb$) across solution paths. We use dataset 548 (dense $1024\times 8192$ matrix, LDR, noERC) and dataset 474 (sparse $8192 \times 49512$ matrix, LDR, extERC) from~\cite{lorenz2015solving}, their most computationally demanding datasets. For each dataset, we first compute the entire solution path $t \mapsto (\bx^{s}(t,\bb),\bp^{s}(t,\bb))$ using \textbf{Algorithm~2} and identify all $K$ kinks in $t$ in the solution path (see Remark~\ref{rem:kinks}). Then we run \textbf{Algorithm~1}, \textbf{glmnet}, \textbf{fista}, and \textbf{mlasso} (for dataset 584 as \textbf{mlasso} does not support sparse matrices) at the kinks $\{t^{(k)}\}_{k=0}^{K}$, starting from $t^{(0)} = \norminf{\Atop\bb}$.

We perform two runs for each dataset. In the first run, we use the default options of \textbf{glmnet} and \textbf{mlasso} and a relative tolerance of $10^{-4}$ for \textbf{fista} (i.e., the algorithm stops when the relative difference in $\ell_{\infty}$-norm of the dual updates is less than $10^{-4}$). In the second run, we use harsher tolerances ($\text{thresh} = 10^{-13}$ for \textbf{glmnet}, $\text{RelTol} = 10^{-8}$ for \textbf{mlasso}, and a relative tolerance of $10^{-8}$ for \textbf{fista}). In both runs, we use a tolerance of $10^{-8}$ in \textbf{Algorithms 1} and \textbf{2} when evaluating their equicorrelation sets on line 3. After convergence of each algorithm, we evaluate the dual objective function in~\eqref{eq:lasso_dual} and number of nonzero components of their primal solutions along the solution path.

Figures~\ref{fig:accuracy_1} and~\ref{fig:accuracy_2} below show the relative errors of the dual objective function and the dual solution with respect to \textbf{Algorithm~1}, as well as the number of nonzero components of the primal solutions near the end of the solution path. \textbf{Algorithm~2} essentially coincides with \textbf{Algorithm~1} in all cases. More importantly, \textbf{Algorithm~1} always achieved better optimality in its dual objective function compared to the other algorithms, hence why we use it as the measure for computing all relative errors.

At default tolerance, \textbf{mlasso} and $\textbf{fista}$ performed reasonably well while \textbf{glmnet} performed poorly. In particular, \textbf{glmnet} produced a numerical solution with many more nonzero components than all other algorithms. This remained the case even in dataset \#~474 with a harsh tolerance. Taken together, we see that \textbf{Algorithm~1} is the clear winner when it comes to accuracy, with \textbf{mlasso} and \textbf{fista} performing well, and with \textbf{glmnet} performing poorly, even when using harsher tolerances.

\begin{figure}
    \centering
    \begin{subfigure}{0.49\textwidth}
        \centering
        \includegraphics[width=0.95\linewidth]{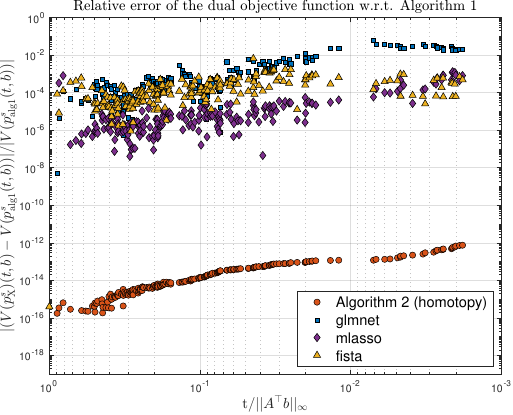}
    \end{subfigure}
    \hfill
    \begin{subfigure}{0.49\textwidth}
        \centering
        \includegraphics[width=0.95\linewidth]{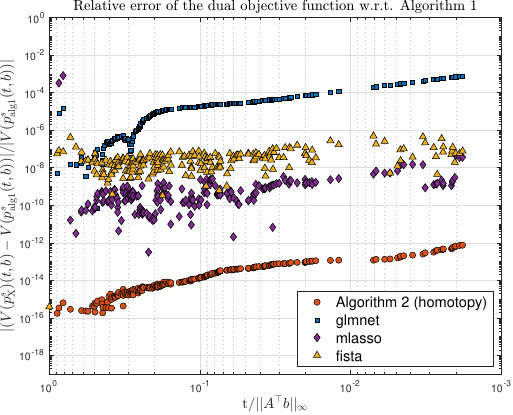}
    \end{subfigure}

    \begin{subfigure}{0.49\textwidth}
        \centering
        \includegraphics[width=0.95\linewidth]{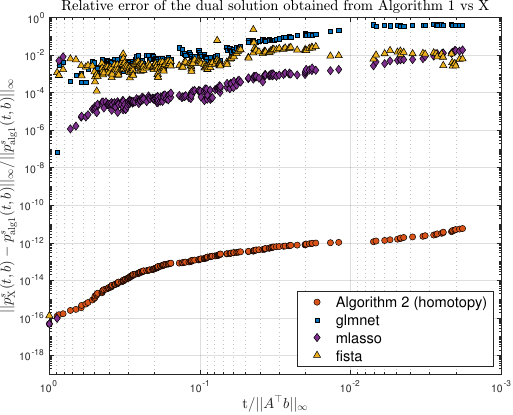}
    \end{subfigure}
    \hfill
    \begin{subfigure}{0.49\textwidth}
        \centering
        \includegraphics[width=0.95\linewidth]{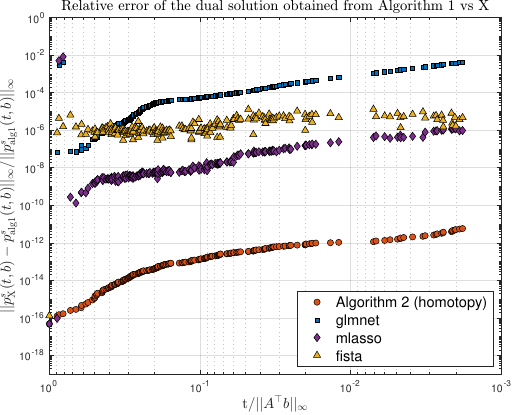}
    \end{subfigure}
    
    \begin{subfigure}{0.49\textwidth}
        \centering
        \includegraphics[width=0.95\linewidth]{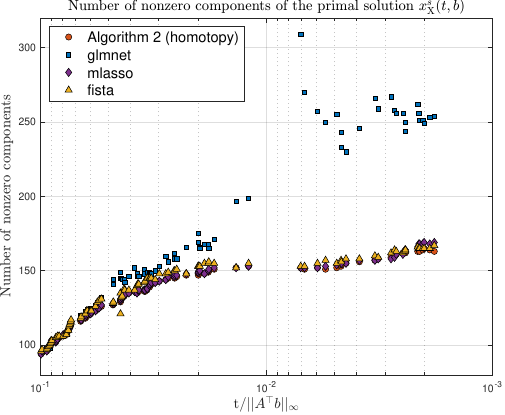}
    \end{subfigure}
    \hfill
    \begin{subfigure}{0.49\textwidth}
        \centering
        \includegraphics[width=0.95\linewidth]{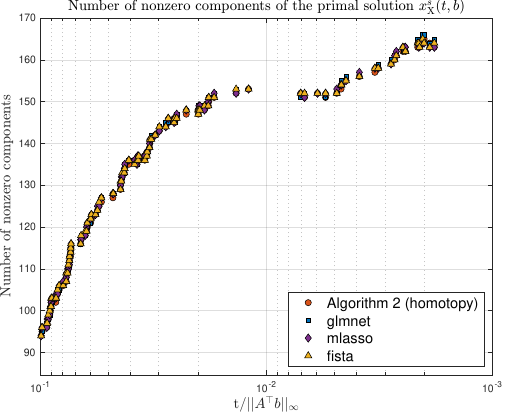}
    \end{subfigure}
    \caption{Relative error of the dual objective function with respect to \textbf{Algorithm~1} (first row), relative error in $\ell_{\infty}$-norm of the dual solution with respect to \textbf{Algorithm~1} (second row), and number of nonzero components of the primal solutions at the end of the solution paths. The dataset used is \#~548 from~\cite{lorenz2015solving}. Left: Default tolerances ($\text{thresh} = 10^{-4}$ for \textbf{glmnet}, $\text{RelTol} = 10^{-4}$ for \textbf{mlasso}, and a relative tolerance of $10^{-4}$ for \textbf{fista}). Right: Harsher tolerances ($\text{thresh} = 10^{-13}$ for \textbf{glmnet}, $\text{RelTol} = 10^{-8}$ for \textbf{mlasso}, and a relative tolerance of $10^{-8}$ for \textbf{fista}).}
    \label{fig:accuracy_1}
\end{figure}
\clearpage
\begin{figure}
    \centering
    \begin{subfigure}{0.49\textwidth}
        \centering
        \includegraphics[width=0.95\linewidth]{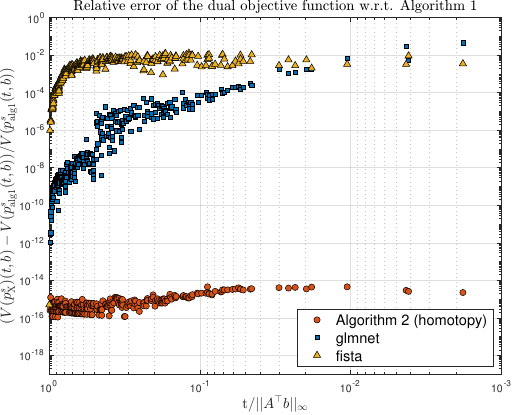}
    \end{subfigure}
    \hfill
    \begin{subfigure}{0.49\textwidth}
        \centering
        \includegraphics[width=0.95\linewidth]{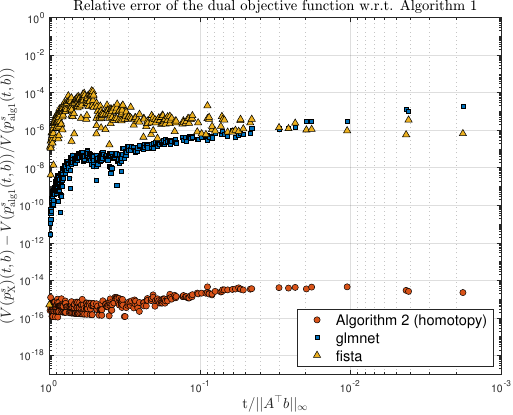}
    \end{subfigure}

    \begin{subfigure}{0.49\textwidth}
        \centering
        \includegraphics[width=0.95\linewidth]{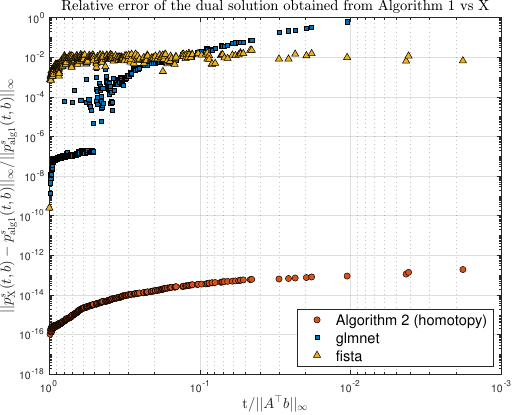}
    \end{subfigure}
    \hfill
    \begin{subfigure}{0.49\textwidth}
        \centering
        \includegraphics[width=0.95\linewidth]{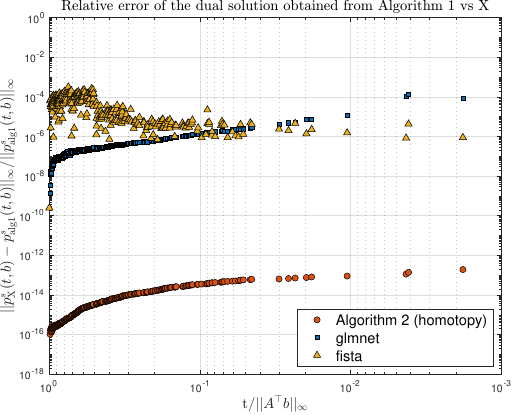}
    \end{subfigure}
    
    \begin{subfigure}{0.49\textwidth}
        \centering
        \includegraphics[width=0.95\linewidth]{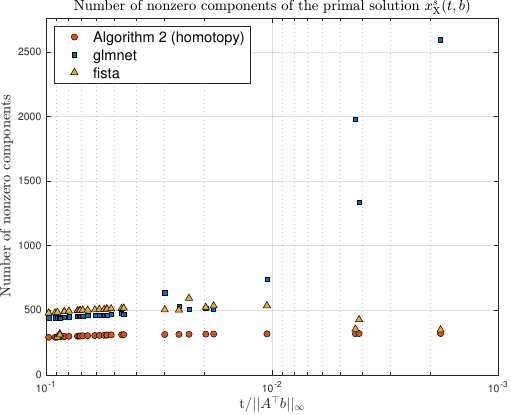}
    \end{subfigure}
    \hfill
    \begin{subfigure}{0.49\textwidth}
        \centering
        \includegraphics[width=0.95\linewidth]{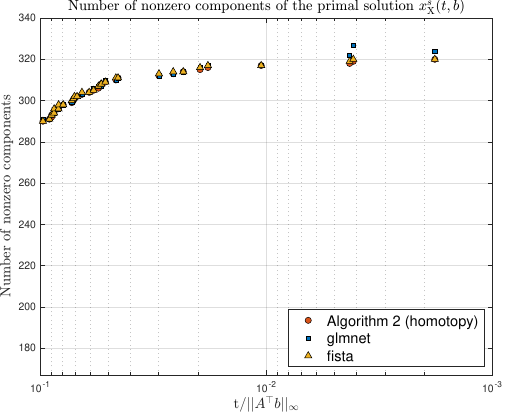}
    \end{subfigure}
    \caption{Relative error of the dual objective function with respect to \textbf{Algorithm~1} (first row), relative error in $\ell_{\infty}$-norm of the dual solution with respect to \textbf{Algorithm~1} (second row), and number of nonzero components of the primal solutions at the end of the solution path. The dataset used is \#~474 from~\cite{lorenz2015solving}. Left: Default tolerances ($\text{thresh} = 10^{-4}$ for \textbf{glmnet}, $\text{RelTol} = 10^{-4}$ for \textbf{mlasso}, and a relative tolerance of $10^{-4}$ for \textbf{fista}). Right: Harsher tolerances ($\text{thresh} = 10^{-13}$ for \textbf{glmnet}, $\text{RelTol} = 10^{-8}$ for \textbf{mlasso}, and a relative tolerance of $10^{-8}$ for \textbf{fista}).}
    \label{fig:accuracy_2}
\end{figure}
\clearpage

\subsection{Run times comparisons}\label{subsec:numerics_runtimes}
\subsubsection{Computing regularization paths}
Table~\ref{tab:timings} shows the run times of \textbf{Algorithm~1}, \textbf{glmnet}, \textbf{mlasso}, and $\textbf{fista}$ for computing regularization paths. The regularization paths were generated using 512/1024 logarithmically spaced points for the dense/sparse matrices over $t \in \norminf{\Atop\bb} \times [10^{-4},1]$. We also included the point $\{0\}$ for \textbf{Algorithm~1}, which the other algorithms cannot handle. Since we wish to compute accurate regularization paths, we used harsh tolerances for each algorithm, just as in Section~\ref{subsec:numerics_accuracy} ($\text{thresh} = 10^{-13}$ for \textbf{glmnet}, $\text{RelTol} = 10^{-8}$ for \textbf{mlasso}, and a relative tolerance of $10^{-8}$ for \textbf{fista}). Table~\ref{tab:timings} shows that the overall winner is \textbf{Algorithm~1}; it is significantly faster than \textbf{fista} and \textbf{mlasso}, and slightly faster or comparable to \textbf{glmnet}.

\begin{table}[ht]
    \centering
    \begin{tabular}{|l|l|l|l|l|l}
    \hline
        Datasets & \textbf{fista} & \textbf{mlasso} & \textbf{glmnet} & \textbf{Algorithm 1}\\
        \hline
        \#~147 (d, HDR, ERC) & 7.15 & 1.63 & 0.87 & 0.65\\
        \hline
        \#~148 (d, HDR, extERC) & 14.46 & 1.60 & 0.86 & 0.69\\
        \hline
        \#~274 (d, HDR, noERC) & 88.70 & 4.16 & 1.01 & 1.22\\
        \hline
        \#~421 (d, LDR, ERC) & 17.13 & 1.57 & 0.88 & 0.58\\
        \hline
        \#~422 (d, LDR, extERC) & 31.69 & 1.62 & 0.89 & 0.60\\
        \hline
        \#~548 (d, LDR, noERC) & 524.21 & 6.64 & 1.82 & 1.06\\
        \hline
        \#~199 (s, HDR, ERC) & 52.49 & N/A & 1.62 & 0.65\\
        \hline
        \#~200 (s, HDR, extERC) & 195.86 & N/A & 1.61 & 1.07\\
        \hline
        \#~473 (s, LDR, ERC) & 479.12 & N/A & 1.57 & 0.50\\
        \hline
        \#~474 (s, LDR, extERC) & 645.76 & N/A & 1.67 & 0.98\\
        \hline
    \end{tabular}
    \caption{Timings (in seconds) for \textbf{Algorithm~1}, \textbf{glmnet}, \textbf{mlasso} and \textbf{fista}. Total times for 512/1024 logarithmically spaced points (d: dense/s: sparse), averaged over 5 runs. \label{tab:timings}}
\end{table}

\subsubsection{Computing optimal solutions when $t = 0$}
Table~\ref{tab:timings2} shows the run times of \textbf{Algorithm~1} with $t=0$ and using as input $-\bb/\norminf{\Atop\bb}$, and \textbf{Algorithm~1} with a regularization path. We used the ten aforementioned benchmark datasets from~\cite{lorenz2015solving}, and the regularization paths for \textbf{Algorithm~1} were generated using 512/1024 logarithmically spaced points for the dense/sparse matrices over $t \in \norminf{\Atop\bb} \times [10^{-4},1]$ and including $\{0\}$. We did not include MATLAB's native linear programming solver because we found that it was slow and often failed to converge. We also did not include commercial solvers for~\eqref{eq:BP_primal}because earlier work by Tendero et al.~\cite{tendero2021algorithm} did similar comparisons with a different version of \textbf{Algorithm~1} (valid for $t=0$ only) and found it superior in performance and accuracy.

Table~\ref{tab:timings2} shows that all the methods have similar running times. In addition, we've verified that all methods correctly computed the solution vectors provided in the benchmark datasets. \textbf{Algorithm~1} with regularization path is often slower because it used a large number of points.
\begin{table}[ht]
    \centering
    \begin{tabular}{|l|l|l|l|l|l}
    \hline
        Datasets  & \textbf{Alg. 1} (Reg. path to $t=0$) & \textbf{Alg. 1} (direct at $t=0$)\\
        \hline
        \#~147 (d, HDR, ERC)  & 0.65 & 0.019   \\
        \hline
        \#~148 (d, HDR, extERC)  & 0.69  & 0.041\\
        \hline
        \#~274 (d, HDR, noERC)  & 1.22  & 0.72 \\
        \hline
        \#~421 (d, LDR, ERC)  & 0.58 & 0.020  \\
        \hline
        \#~422 (d, LDR, extERC)  & 0.60 & 0.035  \\
        \hline
        \#~548 (d, LDR, noERC)  & 1.06 & 0.58  \\
        \hline
        \#~199 (s, HDR, ERC)  & 0.65 & 0.028 \\
        \hline
        \#~200 (s, HDR, extERC)  & 1.07 & 1.23 \\
        \hline
        \#~473 (s, LDR, ERC)  & 0.50 & 0.028  \\
        \hline
        \#~474 (s, LDR, extERC)  & 0.98 & 1.29 \\
        \hline
    \end{tabular}
    \caption{Timings (in seconds) for \textbf{Algorithm~1} with a regularization path and \textbf{Algorithm~1} with $t=0$ using $-\bb/\norminf{\Atop\bb}$ as input. Total times for 512/1024 logarithmically spaced points and including $\{0\}$ (d: dense/s: sparse), averaged over 5 runs. \label{tab:timings2}}
\end{table}
\section{Conclusion and future work}\label{sec:discussion}
In this work, we proved that a minimal selection principle from the theory of differential inclusions (Theorem~\ref{thm:eu_diff_inclusions}) enables one to compute an optimal solution of~\eqref{eq:lasso_dual} from the asymptotic limit of the slow system~\eqref{eq:bp_slow_ode}. As the results in Section~\ref{sec:msp} and~\ref{sec:results} show, the slow system can be integrated, yielding the slow solution~\eqref{eq:slow_global_solution}. The slow solution converges in finite time and, at convergence, yields the optimal solution to~\eqref{eq:lasso_dual}. From it, one can recover an optimal solution to~\eqref{eq:lasso_primal}. Taken together, these results yielded Algorithm~1. We also presented, in Section~\ref{sec:perturbations}, a detailed perturbation analysis of the slow system, including its local and global dependence on the hyperparameter and data. The global continuation of the slow solution provided a rigorous homotopy algorithm for the lasso problem. Our numerical experiments showed that Algorithm~1 vastly outperforms the state of the arts in accuracy while also achieving the best overall performance, highlighting its key feature that it neither compromises accuracy nor computational efficiency.

While this work focused on the lasso problem, our results yielded a novel solution method for solving a broad class of projected dynamical systems. We therefore expect that our results will be relevant to applications involving variational inequalities and projected dynamical systems. In addition, we expect that our results can be adapted to compute exact or approximate solutions to a broader class of convex polyhedral-constrained optimization problems. This includes: (i) variations of the lasso problem where the $\ell_{1}$ norm is replaced by a polyhedral norm or the function $\bx \mapsto \normone{\matr{M}\bx}$ for some appropriate real matrix $\matr{M}$, (ii) other $\ell_{1}$-regularized problems such as logistic regression, Poisson regression, support vector machines and boosting problems, and (iii) extensions to inequality constraints, i.e., linear and quadratic programmming. These problems will be investigated in future work.


\section*{Acknowledgements}
GPL would like to thank Prof.~Xiaodong Wang for useful discussions, his early support and encouragement on this project, and Prof. Georg Stadler and Dr. David K. A. Mordecai for their support and encouragement. This research is supported by ONR N00014-22-1-2667.

\section*{Conflict of interest}
The authors declare that they have no conflict of interest.

\appendix
\section{Mathematical background} \label{app:A}
We list here important definitions and technical results from convex and functional analysis used in this work. For comprehensive references, we refer the reader to~\cite{rockafellar1997convex,rockafellar2009variational,aubin2012differential,ekeland1999convex,hiriart2013convexI,hiriart2013convexII}. All vectors and matrices are denoted in bold typeface. Given an $m\times n$ real matrix $\matr{A}$ and a set of indices $\mathcal{E} \subset \{1,\dots,n\}$, we write $\matr{A}_{\eqset}$ to denote its $m\times|\eqset|$ submatrix with columns indexed by $\mathcal{E}$ and we write $\Atop_{\eqset}$ to denote its transpose. Similarly, given $\bu \in \Rn$, we write $\bu_{\eqset}$ to denote its $|\eqset|$-dimensional subvector indexed by $\eqset$. 

\begin{table}[ht]
    \centering
    \begin{tabular}{l|l}
    \hline
        Notation & Meaning \\
        \hline
        $\{\boldsymbol{e}_{1},\dots,\boldsymbol{e}_{n}\}$ & The set of $n$ canonical vectors of $\Rn$\\
        $\left\langle \bx_{1},\bx_{2} \right\rangle$ & The Euclidean scalar product of two vectors $\bx_{1},\bx_{2} \in \Rn$.\\ 
        $\interior{~C}$ & Interior of a nonempty subset $C$  \\
        $\ri~C$ & Interior of a nonempty subset $C$ relative to the affine hull of $C$ \\
        %
        $\chi_{C}$ & The characteristic function of a set C: \\
        & \quad $\chi_{C}(\bx) \coloneqq \begin{cases} 0, \; &\text{if}\;x \in C \\ +\infty \; &\text{otherwise}\end{cases}$ \\
        %
        %
        $\dom f$ & The domain of a function $f$: $\dom f \coloneqq \left\{ \bx \in \Rn : f(\bx) < +\infty\right\}$ \\
        $\dom \partial f$ & The set of points $\bx \in \dom f$ where the subdifferential $\partial f(\bx) \neq \varnothing$\\
        $f^{*}$& Convex conjugate of a function $f$: $f^{*}(\bs)\coloneqq\sup_{\bx \in \Rn}\left\{\left\langle \bs,\bx\right\rangle -f(\bx)\right\}$ \\
        $\proj_{C}(\bx)$ & Projection of $\bx \in \Rn$ onto a closed convex set $C \subset \Rn$: \\
        & \quad $\proj_{C}(\bx) \coloneqq \argmin_{\by\in C}\left\Vert \bx-\by\right\Vert _{2}^{2}$ \\
        \hline
    \end{tabular}    \label{tab:notation}
\end{table}


        
        

\subsection*{Definitions}
\begin{defn}[Convex sets] \label{def:convex_sets}
    A subset $C\subset \Rn$ is convex if for every pair $(\bx_{1},\bx_{2}) \in C \times C$ and every scalar $\lambda \in (0,1)$, the point $\lambda\bx_{1}+(1-\lambda)\bx_{2}$ is contained in $C$.
\end{defn}
\begin{defn}[Closed convex polyhedra]\label{def:convex_polyhedron}
    A nonempty set $C \subset \Rn$ is a closed convex polyhedron if it can be expressed as $C \coloneqq \left\{\bx \in \Rn : \langle\bu_{j},\bx\rangle \leqslant r_{j} \ \text{for every $j \in \{1,\dots,l\}$}  \right\}$, where $\{\bu_{1},\dots,\bu_{l}\} \subset \Rn$ and $\{r_{1},\dots,r_{l}\}\subset \R$.
\end{defn}
\begin{defn}[Convex cones]\label{def:convex_cones}
    A nonempty set $K \subset \Rn$ is a cone if $\boldsymbol{0} \in K$ and $\lambda\bx \in K$ for all $\bx \in K$ and $\lambda > 0$. A cone $K$ is convex if it contains the point $\sum_{j=1}^{k} \eta_{j}\bx_{j}$ whenever $\bx_{j} \in K$ and $\eta_{j} \geqslant 0$ for $j \in \{1,\dots,k\}$. 
\end{defn}
\begin{defn}[Conical hulls]\label{def:conical_hulls}
    The conical hull of $k$ vectors $\{\bx_{1},\dots,\bx_{k}\} \subset \Rn$ is defined as the closed convex cone
    \[
    \cone\{\bx_{1},\dots,\bx_{k}\} \coloneqq \left\{\sum_{j=1}^{k}\eta_{j}\bx_{j} : \eta_{j} \geqslant 0 \, \text{for every}\ j\in\{1,\dots,k\}\right\}.
    \]
\end{defn}
\begin{defn}[Polyhedral cones]
    A nonempty cone $K \subset \Rn$ is polyhedral if it can be expressed as $K = \cone\{\bx_{1},\dots,\bx_{k}\}$ for some finite collection of vectors $\{\bx_{1},\dots,\bx_{k}\} \subset \Rn$.
\end{defn}
\begin{defn}[Proper functions] \label{def:prop_f}
    A function $f$ defined over $\Rn$ is proper if its domain $\dom f \coloneqq \left\{ \bx \in \Rn : f(\bx) < +\infty\right\}$ is nonempty and $f(\bx) > -\infty$ for every $\bx \in \dom f$.
\end{defn}
\begin{defn}[Lower semicontinuity] \label{def:lsc_f}
    A proper function $f\colon\Rn\to\R\cup\{+\infty\}$ is lower semicontinuous at $\bx \in \dom f$ if for every sequence $\left\{ \bx_k\right\} _{k=1}^{+\infty} \subset \Rn$ converging to $\bx$, $\liminf_{k\to+\infty}f(\bx_k)\geqslant f(\bx)$. We say that $f$ is lower semicontinuous if it is lower semicontinuous at every $\bx \in \dom f$.
\end{defn}
\begin{defn}[Convex functions] \label{def:convex} 
A proper function $f\colon\Rn\to\R\cup\{+\infty\}$ is convex if its domain $\dom f$ is convex and if for every pair $(\bx_{1},\bx_{2}) \in \dom f \times \dom f$ and every $\lambda\in[0,1]$,
\[
f(\lambda\bx_{1}+(1-\lambda)\bx_{2})\leqslant\lambda f(\bx_{1})+(1-\lambda)f(\bx_{2}).
\]
It is \textit{strictly convex} if the inequality above is strict whenever $\bx_{1}\neq\bx_{2}$ and $\lambda \in (0,1)$, and it is \textit{$t$-strongly convex} with $t>0$ if for every pair $(\bx_{1},\bx_{2}) \in \dom f \times \dom f$ and every $\lambda\in[0,1]$, 
\[
f(\lambda\bx_{1}+(1-\lambda)\bx_{2})\leqslant\lambda f(\bx_{1})+(1-\lambda)f(\bx_{2})-\frac{t}{2}\lambda(1-\lambda)\normsq{\bx_{1}-\bx_{2}}.
\]
\end{defn}
\begin{defn}[The set $\gmRn$]\label{def:gmrn}
The space of proper, lower semicontinuous and convex functions over $\Rn$ is denoted by $\gmRn$.
\end{defn}
\begin{defn}[Differentiability] \label{def:differentiability}
Let $f$ be a proper function with $\interior{~\dom f} \neq \emptyset$. The function $f$ is differentiable at $\bx \in \interior{~\dom f}$ if there is $\bs \in \Rn$ such that for every $\bd \in \Rn$, $f'(\bx,\bd) = \left\langle \bs,\by\right\rangle$. If $\bs$ exists, then it is unique, it is called the gradient of $f$ at $\bx$, and it is written as $\bs \equiv \nabla f(\bx)$.
\end{defn}
\begin{defn}[Subdifferentiability and subgradients]\label{def:subgrad}
A function $f \in \gmRn$ is subdifferentiable at $\bx \in \Rn$ if there exists $\bs \in \Rn$ such that for every $\by \in \dom f$,
\begin{equation}\label{eq:conv_subdiff_char}
f(\by) \geqslant f(\bx) + \left\langle\bs,\by-\bx \right\rangle.
\end{equation}
In this case, $\bs$ is called a subgradient of the function $f$ at $\bx$. The set of subgradients at $\bx \in \Rn$ is called the subdifferential of $f$ at $\bx$ and is denoted by $\partial f(\bx)$. Moreover:
\begin{itemize}
    \item[(i)] When nonempty, the subdifferential of $f$ at $\bx$ is a closed convex set. The set of points $\bx \in \dom f$ for which $\partial f(\bx)$ is nonempty is denoted by $\dom \partial f$.
    \item[(ii)] The function $f$ has a unique subgradient at $\bx$ if and only if $f$ is differentiable at $\bx$ and $\partial f(\bx) = \{\nabla f(\bx)\}$~\cite[Proposition 5.3, page 23]{ekeland1999convex}.

    \item[(iii)] If $f$ is strictly convex, then the inequality in~\eqref{eq:conv_subdiff_char} is strict whenever $\bx \neq \by$. If $f$ is $t$-strongly convex with $t>0$, then $f$ is subdifferentiable at $\bx \in \Rn$ if there exists $\bs \in \Rn$ such that for every $\by \in \dom f$,
    \begin{equation}\label{eq:sc_subdiff_char}
    f(\by) \geqslant f(\bx) + \left\langle \bs,\by-\bx\right\rangle + \frac{t}{2}\normsq{\bx-\by}.
    \end{equation}
\end{itemize}
\end{defn}
\begin{defn}[Monotone and maximal monotone mappings]\label{def:mmm}
Let $F$ denote a set-valued mapping from $\Rn$ to $\Rn$ with \textit{graph} $\{(\bx,\bv) \in \Rn \times \Rn \colon \bv \in F(\bx)\}$. The set-valued mapping $F$ is \textit{monotone} if for every $\bx_{1},\bx_{2} \in \Rn$ and every $\bv_{1} \in F(\bx_{1})$, $\bv_{2} \in F(\bx_{2})$,
\begin{equation}\label{eq:monotone_inequality}
\langle\bv_{1}-\bv_{2},\bx_{1}-\bx_{2}\rangle \geqslant 0.
\end{equation}
It is \textit{maximal} if no other set-valued mapping $\tilde{F}$ contains strictly the graph of $F$.

If $f \in \gmRn$, then its subdifferential is a maximal monotone mapping~\cite[Theorem 12.17]{rockafellar2009variational}. If also $f$ is $t$-strongly convex with $t>0$, then for every $\bx_{1},\bx_{2} \in \Rn$ and every $\bv_{1} \in \partial f(\bx_{1})$, $\bv_{2} \in \partial f(\bx_{2})$ we have the stronger monotone inequality~\cite[Corollary 31.5.2]{rockafellar1997convex},
\begin{equation}\label{eq:stronger_monotone_inequality}
\langle\bv_{1}-\bv_{2},\bx_{1}-\bx_{2}\rangle \geqslant t\normsq{\bx_{1}-\bx_{2}}.
\end{equation}
\end{defn}
\begin{defn}[Descent directions]\label{def:descent_dir}
A vector $\bd \in \Rn$ is a descent direction for a proper function $f$ at $\bx \in \dom f$ if there exists $\tau > 0$ such that $\bx + \tau\bd \in \dom f$ and $f(\bx + \tau\bd) < f(\bx)$.
\end{defn}
\begin{defn}[Convex conjugates] \label{def:conv_conj}
Let $f \in \gmRn$. The convex conjugate $f^{*}\colon\Rn\to\mathbb{R}\cup\{+\infty\}$ of $f$ is defined by $f^{*}(\bs)\coloneqq\sup_{\bx \in \Rn}\left\{\left\langle \bs,\bx\right\rangle -f(\bx)\right\}$. In particular, $f^* \in \gmRn$~\cite[Definition 4.1]{ekeland1999convex}.
\end{defn}
\begin{defn}[Coercive functions]\label{def:coercive}
A proper function $f\colon\Rn\to\R\cup\{+\infty\}$ is coercive if for every sequence $\{\bx_k\}_{k=1}^{+\infty} \subset \Rn$ with $\lim_{k \to +\infty} \normtwo{\bx_k} = +\infty$, we have $\lim_{k \to +\infty} f(\bx_k) = +\infty$. If $f \in \gmRn$, then $f$ is coercive if and only if $\boldsymbol{0} \in \ri \, \dom f^{*}$~\cite[Theorem 11.8, page 479]{rockafellar2009variational}.
\end{defn}
\begin{defn}[Characteristic functions]\label{def:char_fun}
The characteristic function of a nonempty subset $C \subset \Rn$ is defined as
\[
\chi_{C}(\bx) = \begin{cases} 0, \; &\text{if}\;x \in C \\ +\infty \; &\text{otherwise}\end{cases}.
\]
We have $\chi_{C} \in \gmRn$ if and only if $C$ is nonempty, closed and convex, and it is coercive if and only if $C$ is bounded.
\end{defn}
\begin{defn}[Euclidean projection]\label{def:projection}
Let $C$ be a nonempty, closed and convex subset of $\Rn$ and let $\bx \in \Rn$. The projection of $\bx$ on $C$ is $\proj_{C}(\bx) \coloneqq \argmin_{\by \in C} \normtwo{\bx-\by}$. The projection of $\bx$ on $C$ exists, is unique, and satisfies the characterization
\begin{equation}\label{eq:char_projections}
\langle \bx - \proj_{C}(\bx), \by - \proj_{C}(\bx) \rangle \leqslant 0 \; \text{for every } \by \in C.
\end{equation}
See~\cite[Section III, 3.1]{hiriart2013convexI} for details.
\end{defn}
\begin{defn}[Normal cones]\label{def:normal_cones}
    Let $C$ be a nonempty closed convex set. The subdifferential of the characteristic function $\chi_{C}$ at $\bx \in C$ is the set of normal vectors
    \[
    \partial\chi_{C}(\bx) \coloneqq \left\{\bs \in \Rn : \langle\bs,\bx-\by\rangle \geqslant 0 \ \text{for every $\by \in C$} \right\}.
    \]
    This set is called the normal cone of $C$ at $\bx \in C$, and it is a closed convex cone.
\end{defn}
%
%
\subsection*{Technical results}
\begin{thm}[The generalized Fermat's rule]\label{thm:fermat}
Let $f \in \gmRn$. Then $f$ has a global minimum at $\bxsol$ if and only if $\boldsymbol{0} \in \partial f(\bxsol)$.    
\end{thm}
\begin{proof}
    See~\cite[Theorem 10.1, page 422]{rockafellar2009variational} for the proof.
\end{proof}
\begin{thm}[Fenchel's inequality]
Let $f \in \gmRn$. For every $\bx, \bs \in \Rn$, the function $f$ and its convex conjugate $f^{*}$ satisfy Fenchel's inequality:
\begin{equation} \label{eq:fenchel_ineq}
f(\bx)+f^{*}(\bs) \geqslant \left\langle \bs,\bx\right\rangle,
\end{equation}
with equality if and only if $\bs\in\partial f(\bx)$, if and only if $\bx\in\partial f^{*}(\bs)$.
\begin{proof}
See \cite[Corollary 1.4.4, page 48]{hiriart2013convexII}.
\end{proof}
\end{thm}
\begin{prop}[The subdifferential of the characteristic function of a closed convex polyhedron]\label{prop:subdiff_ccp} Let $C \subset \Rn$ denote the nonempty, closed convex polyhedron 
\[C \coloneqq \left\{\bx \in \Rn : \langle\bu_{j},\bx\rangle \leqslant r_{j} \ \text{for every $j \in \{1,\dots,l\}$}  \right\},
\]
where $\{\bu_{1},\dots,\bu_{l}\} \subset \Rn$ and $\{r_{1},\dots,r_{l}\}\subset \R$. Let $\mathcal{E}(\bx) = \{j\in\{1,\dots,l\} : \langle\bu_{j},\bx\rangle = r_{j}\}$ denote the set of active constraints at $\bx \in \Rn$. The subdifferential of the characteristic function of $C$ at $\bx$ is the polyhedral cone
\[
\partial\chi_{C}(\bx) = \left\{\sum_{j \in \mathcal{E}(\bx)}\eta_{j}\bu_{j} : \eta_{j} \geqslant 0\right\}.
\]
\end{prop}
\begin{proof}
This follows from Definitions~\ref{def:convex_polyhedron} and~\ref{def:normal_cones}. See~\cite[Example 5.2.6]{hiriart2013convexI} for details.
\end{proof}
\begin{prop}[Properties of subdifferentials and some calculus rules]\label{prop:precomposition}
\noindent
\begin{itemize}
    \item[(i)] Let $f_{1},f_{2} \in \gmRn$ and assume $\ri \, \dom f_{1} \cap \ri \, \dom f_{2} \neq \emptyset$. Then $f_{1} + f_{2} \in \gmRn$ and for any $\bx \in \dom f_{1} \cap \dom f_{2}$, we have $\partial(f_{1} + f_{2})(\bx) = \partial f_{1}(\bx) + \partial f_{2}(\bx)$.

    \item[(ii)] Let $f_{1},f_{2} \in \gmRn$ and assume $\ri \, \dom f_{1} \cap \ri \, \dom f_{2} \neq \emptyset$. Then for any $\bs \in \dom f_{1}^{*} + \dom f_{2}^{*}$, we have $(f_{1} + f_{2})^{*}(\bs) = \inf_{\bs_{1} + \bs_{2} = \bs}\left\{f_{1}^{*}(\bs_{1}) + f_{2}^{*}(\bs_{2})\right\}$.

    \item[(iii)] Let $f \in \gmRn$, $g \in \gmRm$ and assume $\image{\matr{A}} \cap \ri \, \dom g \neq \emptyset$. Then the function $\bx \mapsto g(\matr{A}\bx)$ is in $\gmRn$, $\dom(g \circ \matr{A})^{*} = \matr{A}^{\top}\dom g^{*}$, and for all $\bx \in \Rn$ satisfying $\matr{A}\bx \in \dom g$ we have $\partial(g \circ \matr{A})(\bx) = \matr{A}^{\top}\partial g(\matr{A}\bx)$.
    
    \item[(iv)] Let $f \in \gmRn$, $g \in \gmRm$ and assume $\ri \, \dom f \cap \ri \, \dom (g \circ \matr{A}) \neq \emptyset$. Then for any $\bx \in \dom f \cap \dom (g\circ\matr{A})$, we have $\partial(f + g\circ\matr{A}) = \partial f(\bx) + \matr{A}^{\top}\partial g(\matr{A}\bx)$.
\end{itemize}
\begin{proof}
See~\cite[Corollary 3.1.2]{hiriart2013convexII} for the proof of (i),~\cite[Theorem 2.3.3]{hiriart2013convexII} for the proof of (ii), and~\cite[Theorem 2.2.1, Theorem 2.2.3, and Theorem 3.2.1.]{hiriart2013convexII} for the proof of (iii). The proof of (iv) follows by using (i) with $f_{1} = f$ and $f_{2} = g \circ \matr{A}$ and then using (iii).
\end{proof}
\end{prop}
\begin{thm}[The primal problem and the dual problem]\label{thm:primal_dual_prob}
Let $f_{1} \in \gmRn$, $f_{2} \in \gmRm$ and let $\matr{A}$ denote a real $m \times n$ matrix. Consider the ``primal" minimization problem
\begin{equation}\label{eq:fact:primal_prob}
\inf_{\bx \in \Rn} \left\{f_{1}(\bx) + f_{2}(\matr{A}\bx)\right\}.
\end{equation}
Assume
\begin{equation}\label{eq:assumptions_eu}
\boldsymbol{0} \in \ri(\matr{A}~\dom f_{1} - \dom f_{2}) \quad \text{and} \quad \boldsymbol{0} \in \ri(\dom f_{1}^{*} + \matr{A}^{\top}\dom f_{2}^{*}).
\end{equation}
Then:
\begin{itemize}
    \item[(i)] The primal problem~\eqref{eq:fact:primal_prob} has at least one solution. If $\bxsol$ denote such a solution, then it satisfies the first-order optimality condition
    \begin{equation}\label{eq:fo_primal}
     \boldsymbol{0} \in \partial(f_{1} + f_{2}\circ \matr{A})(\bxsol) \quad \iff \quad \boldsymbol{0} \in \partial f_{1}(\bxsol) + \matr{A}^{\top}\partial f_{2}(\matr{A}\bxsol).
    \end{equation}
    \item[(ii)] The ``dual" maximization problem
\begin{equation}\label{eq:fact:dual_prob}
\sup_{\bp \in \Rm} \left\{-f_{1}^{*}(-\Atop\bp)- f_{2}^{*}(\bp) \right\}
\end{equation}
has at least one solution and is equal in value to the primal problem~\eqref{eq:fact:primal_prob}. If $\bpsol$ denote such a solution, then it satisfies the first-order optimality condition
\begin{equation}\label{eq:fo_dual}
    \boldsymbol{0} \in \partial(f_{1}^{*}\circ(-\matr{A}^{\top}) + f_{2}^{*})(\bpsol)\quad \iff \quad \boldsymbol{0} \in  -\matr{A}\partial f_{1}^{*}(-\matr{A}^{\top}\bp^s) + \partial f_{2}^{*}(\bp^s) .
\end{equation}
\item[(iii)] The optimality conditions~\eqref{eq:fo_primal} and~\eqref{eq:fo_dual} are equivalent and can be written as
\begin{equation}\label{eq:fo_joined}
    \bpsol \in \partial f_{2}(\matr{A}\bxsol) \ \text{and} \ -\matr{A}^\top\bpsol \in \partial f_{1}(\bxsol).
\end{equation}
\end{itemize}
\begin{proof}
See~\cite[Theorem 31.2 and Corollary 31.2.1]{rockafellar1997convex} and \cite[Proposition 1, Theorem 1 and Theorem 2, pages 163--167]{aubin2012differential} for a proof. The equivalences in~\eqref{eq:fo_primal} and~\eqref{eq:fo_dual} follow from~\eqref{eq:assumptions_eu} and the rules in Proposition~\ref{prop:precomposition}. The equivalence in~\eqref{eq:fo_joined} follows from Fenchel's inequality~\eqref{eq:fenchel_ineq}.
\end{proof}
\end{thm}
\begin{rem}\label{rem:coercivity}
The assumptions in~\eqref{eq:assumptions_eu} read: there exists $\bx \in \dom f_{1}$ such that $\matr{A}\bx \in \ri \ \dom f_{2}$ and there exists $\bp \in \dom f_{2}^{*}$ such that $-\Atop\bp \in \dom f_{1}^{*}$. If the first assumption holds, then Proposition~\ref{prop:precomposition} implies the function $\bx \mapsto f_{1}(\bx) + f_{2}(\matr{A}\bx)$ is in $\gmRn$. Moreover, 
\[
\ri \, \dom (f_{1} + (f_{2} \circ \matr{A}))^{*} = \ri(\dom f_{1}^{*} + \dom(f_{2} \circ \matr{A})^{*})
= \ri(\dom f_{1}^{*} + \matr{A}^{\top}\dom f_{2}^{*}).
\] 
In particular, the second assumption becomes equivalent to coercivity of the function 
$\bx \mapsto f_{1}(\bx) + f_{2}(\matr{A}\bx)$ (see Definition~\ref{def:coercive}). 
\end{rem}
%
	
\section{Technical Proofs}\label{app:detailed_proofs}
\subsection{Proof of Proposition~\ref{prop:descent_direction}}\label{app:proof_descent}
\textbf{Part 1.} First, we show there is $\Delta_{*}(\bpo;t,\bb) > 0$ such that $\norminf{-\Atop(\bpo + \Delta\bd(\bpo;t,\bb))} \leqslant 1$ for every $\Delta \in [0,\Delta_{*}(\bpo;t,\bb)]$. Multiply the vector $-\Atop(\bpo + \Delta\bd(\bpo;t,\bb))$ by the matrix of signs $\matr{D}(\bpo)$ and take the inner product with respect to the unit vector $\be_{j}$ with $j \in \{1,\dots,n\}$:
\begin{equation}\label{eq:app_part_case1}
\langle-\matr{D}(\bpo)\Atop(\bpo + \Delta\bd(\bpo;t,\bb)),\be_j\rangle = \langle -\matr{D}(\bpo)\Atop\bpo,\be_j\rangle - \Delta\langle \matr{D}(\bpo)\Atop\bd(\bpo;t,\bb),\be_j\rangle.
\end{equation}
By definition of the equicorrelation set~\eqref{eq:equicorrelation_set} and the matrix of signs $\matr{D}(\bpo)$, we have
\[
0 \leqslant \langle -\matr{D}(\bpo)\Atop\bpo,\be_j\rangle \leqslant 1 \ \text{for every $j \in \{1,\dots,n\}$.}
\]
We now proceed according to whether $j \in \epo$ or $j \in \ecpo$.

First, suppose $j \in \epo$. Then $\langle -\matr{D}(\bpo)\Atop\bpo,\be_j\rangle = 1$ and $\langle \matr{D}(\bpo)\Atop\bd(\bpo;t,\bb),\be_j\rangle \geqslant 0$ by the KKT condition~\eqref{eq:nnls_problem_oc_2}. If the latter is strictly positive, then we can increase $\Delta$ until the left-hand-side of~\eqref{eq:app_part_case1} is equal to $-1$. The smallest such number is
\[
\Delta_{\epo} = \min_{j \in \epo} 2/\langle \matr{D}(\bpo)\Atop\bd(\bpo;t,\bb),\be_j\rangle,
\]
which may be the extended value $\{+\infty\}$ if $\langle \matr{D}(\bpo)\Atop\bd(\bpo;t,\bb),\be_j\rangle = 0$ for every $j \in \epo$.

Next, suppose $j \in \ecpo$ and $\langle \matr{D}(\bpo)\Atop\bd(\bpo;t,\bb),\be_j\rangle \geqslant 0$. The same reasoning as above shows
\[
\Delta_{\ecpo,\geqslant} = \min_{\substack{j \in \ecpo\\ \langle \matr{D}(\bpo)\Atop\bd(\bpo;t,\bb),\be_j\rangle \geqslant 0}} \frac{1 - \langle \matr{D}(\bpo)\Atop\bpo,\be_j\rangle}{\langle \matr{D}(\bpo)\Atop\bd(\bpo;t,\bb),\be_j\rangle}
\]
is the smallest number for which $\langle \matr{D}(\bpo)\Atop\bd(\bpo;t,\bb),\be_j\rangle \geqslant 0$ among $j \in \ecpo$, which may be the extended value $\{+\infty\}$ if $\langle \matr{D}(\bpo)\Atop\bd(\bpo;t,\bb),\be_j\rangle = 0$ for every $j \in \ecpo$.

Finally, suppose $j \in \ecpo$ and $\langle \matr{D}(\bpo)\Atop\bd(\bpo;t,\bb),\be_j\rangle < 0$. Then we can increase $\Delta$ until the left-hand-side of~\eqref{eq:app_part_case1} is equal to $1$. The smallest such number is
\[
\Delta_{\ecpo,<} = \min_{\substack{j \in \ecpo\\ \langle \matr{D}(\bpo)\Atop\bd(\bpo;t,\bb),\be_j\rangle < 0}} \frac{1 + \langle \matr{D}(\bpo)\Atop\bpo,\be_j\rangle}{|\langle \matr{D}(\bpo)\Atop\bd(\bpo;t,\bb),\be_j\rangle|}.
\]

Hence we find $\norminf{-\Atop(\bpo + \Delta\bd(\bpo;t,\bb))} \leqslant 1$ for all $\Delta \in [0,\Delta_{*}(\bpo;t,\bb)]$, where 
\[
\begin{alignedat}{1}
\Delta_{*}(\bpo;t,\bb) &\coloneqq \min{(\Delta_{\epo},\Delta_{\ecpo,\geqslant},\Delta_{\ecpo,<})} \\
&\equiv \min_{j \in \{1,\dots,n\}} \left\{\frac{\sign{\left\langle \matr{D}(\bpo)\Atop\bd(\bpo;t,\bb),\be_j\right\rangle} - \left\langle \matr{D}(\bpo)\Atop\bpo,\be_j\right\rangle)}{\left\langle\matr{D}(\bpo) \Atop\bd(\bpo;t,\bb),\be_j\right\rangle} \right\}.
\end{alignedat}
\]

Now, for the equivalence, note the assumption $\rank{\matr{A}} = m$ implies $\Atop\bd(\bpo;t,\bb) = \boldsymbol{0} \iff \bd(\bpo;t,\bb) = \boldsymbol{0}$. Clearly $\bd(\bpo;t,\bb) = \boldsymbol{0}$ implies $\Delta_{*}(\bpo;t,\bb) = +\infty$. Finally, $\Delta_{*}(\bpo;t,\bb) = +\infty$ only if $\Atop\bd(\bpo;t,\bb) = \boldsymbol{0}$, which is equivalent to $\bd(\bpo;t,\bb) = \boldsymbol{0}$.

\noindent
\newline
\textbf{Part 2.} Using Part 1, we find $V(\bpo + \Delta\bd(\bpo;t,\bb); t,\bb) < +\infty$ for every $\Delta \in [0,\Delta_{*}(\bpo;t,\bb)]$. In particular, we can write
\[
\begin{alignedat}{1}
V(\bpo + \Delta\bd(\bpo;t,\bb); t,\bb) &= \frac{t}{2}\normsq{\bpo + \Delta\bd(\bpo;t,\bb)} + \langle\bb,\bpo + \Delta\bd(\bpo;t,\bb)\rangle \\
&= t \Delta\langle\bpo,\bd(\bpo;t,\bb)\rangle + \Delta^{2}t \normsq{\bd(\bpo;t,\bb)}/2 \\
&\qquad\qquad + \Delta\langle\bb,\bd(\bpo;t,\bb)\rangle + \frac{t}{2}\normsq{\bpo} + \langle\bb,\bpo\rangle \\
&= t \Delta\langle\bpo,\bd(\bpo;t,\bb)\rangle + \Delta^{2}t \normsq{\bd(\bpo;t,\bb)}/2 \\
&\qquad \qquad + \Delta\langle\bb,\bd(\bpo;t,\bb)\rangle + V(\bpo;t,\bb).
\end{alignedat}
\]
Substituting~\eqref{eq:nnls_problem_oc_3} in the above and rearranging yields
\[
V(\bpo + \Delta\bd(\bpo;t,\bb); t,\bb) - V(\bpo;t,\bb) = \Delta(t \Delta/2 - 1)\normsq{\bd(\bpo;t,\bb)}.
\]
This proves Equation~\eqref{eq:V_formula_descent}. Finally, suppose $\bd(\bpo;t,\bb) \neq \boldsymbol{0}$, meaning $0 < \Delta_{*}(\bpo;t,\bb) < +\infty$. Taking $\Delta \in (0,\min(\Delta_{*}(\bpo;t,\bb),2/t))$ in~\eqref{eq:V_formula_descent} yields $V(\bpo + \Delta\bd(\bpo;t,\bb); t,\bb) - V(\bpo;t,\bb) < 0$. Hence $\bd(\bpo;t,\bb)$ is a descent direction.

\noindent
\newline
\textbf{Part 3.} Let $\Delta \in [0,\Delta_{*}(\bpo;t,\bb)]$ and $\bq \in \Rm$. Since $-\bd(\bpo;t,\bb) \in \partial_{\bp} V(\bpo;t,\bb)$ and $V(\cdot;t,\bb)$ is $t$-strongly convex, the subdifferentiability property implies
\[
\begin{alignedat}{1}
V(\bq;t,\bb) &\geqslant V(\bpo;t,\bb) + \langle-\bd(\bpo;t,\bb), \bq - \bpo\rangle + \frac{t}{2}\normsq{\bq-\bpo}\\
&= V(\bpo;t,\bb) + \langle-\bd(\bpo;t,\bb),\bq - (\bpo + \Delta\bd(\bpo;t,\bb))\rangle + \langle-\bd(\bpo;t,\bb),\Delta\bd(\bpo;t,\bb)\rangle \\
&\qquad + \frac{t}{2}\normsq{\bq - (\bpo + \Delta\bd(\bpo;t,\bb)) + \Delta\bd(\bpo;t,\bb)} \\
&= V(\bpo;t,\bb) + \langle-\bd(\bpo;t,\bb),\bq - (\bpo + \Delta\bd(\bpo;t,\bb))\rangle - \Delta\normsq{\bd(\bpo;t,\bb)}\\
&\qquad + \frac{t}{2}\normsq{\bq - (\bpo + \Delta\bd(\bpo;t,\bb))} + \frac{\Delta^2 t}{2}\normsq{\bd(\bpo;t,\bb)} \\
&\qquad\qquad + t \Delta\langle\bd(\bpo;t,\bb),\bq - (\bpo + \Delta\bd(\bpo;t,\bb))\rangle. \\
\end{alignedat}
\]
Use Equation~\eqref{eq:V_formula_descent} in the inequality above to simplify the right hand side:
\[
\begin{alignedat}{1}
V(\bq;t,\bb) &\geqslant V(\bpo + \Delta\bd(\bpo;t,\bb);t,\bb) + \langle-(1-t \Delta)\bd(\bpo;t,\bb),\bq - (\bpo + \Delta\bd(\bpo;t,\bb))\rangle \\
&\qquad\qquad + \frac{t}{2}\normsq{\bq - (\bpo + \Delta\bd(\bpo;t,\bb))}.
\end{alignedat}
\]
By definition of subdifferentiability, this means
\[
-(1-t \Delta)\bd(\bpo;t,\bb) \in \partial_{\bp} V(\bpo + \Delta\bd(\bpo;t,\bb);t,\bb) \ \text{for every $\Delta \in [0,\Delta_{*}(\bpo;t,\bb)]$}.
\]

\noindent
\newline
\textbf{Part 4.} The inclusions and identity follow for $\Delta = 0$, so suppose $\Delta \in (0,\Delta_{*}(\bpo;t,\bb))$. Assume $j \notin \epo$. By construction from Part 1, we have $j \notin \mathcal{E}(\bpo + \Delta\bd(\bpo;t,\bb))$, hence the contrapositive $\mathcal{E}(\bpo + \Delta\bd(\bpo;t,\bb)) \subset \epo$.

Now, let $\bs \in \partial_{\bp}V(\bpo + \Delta\bd(\bpo;t,\bb);t,\bb)$. Then there is $\hat{\bu} \in \Rn$ with $\hat{\bu}_{\epdo} \geqslant \boldsymbol{0}$ and $\hat{\bu}_{\ecpdo} = \boldsymbol{0}$ such that
\[
\bs = \bb + t(\bpo + \Delta\bd(\bpo;t,\bb)) - \matr{A}_{\epdo}\matr{D}_{\epdo}\hat{\bu}_{\epdo}.
\]
In particular,
\[
\bs - t \Delta\bd(\bpo;t,\bb) = \bb + t\bpo  - \matr{A}_{\epdo}\matr{D}_{\epdo}\hat{\bu}_{\epdo}.
\]
From the inclusion $\mathcal{E}(\bpo + \Delta\bd(\bpo;t,\bb)) \subset \epo$, we find $\bs \in \{t \Delta\bd(\bpo;t,\bb)\} + \partial_{\bp}V(\bpo;t,\bb)$. Since $\bs$ was arbitrary, we deduce the inclusion
\[
\partial_{\bp}V(\bpo + \Delta\bd(\bpo;t,\bb);t,\bb) \subset \{t\Delta\bd(\bpo;t,\bb)\} + \partial_{\bp}V(\bpo;t,\bb).
\]

\noindent
\newline
\textbf{Part 5.} Next, we prove identity~\eqref{eq:identity_dd}. Let $\bs = \proj_{\partial_{\bp}V(\bpo+\Delta\bd(\bpo;t,\bb);t,\bb)}(\boldsymbol{0})$ and use both the previous inclusion and subdifferentiability to find
\[
\begin{alignedat}{1}
V(\bpo + \Delta\bd(\bpo;t,\bb);t,\bb) &\geqslant V(\bpo;t,\bb) + \langle\bs - t \Delta\bd(\bpo;t,\bb),\bpo + \Delta\bd(\bpo;t,\bb)-\bpo\rangle\\
&\qquad\qquad  + \frac{t}{2}\normsq{\bpo + \Delta\bd(\bpo;t,\bb) - \bpo} \\
&= V(\bpo;t,\bb) + \langle\bs - t \Delta\bd(\bpo;t,\bb), \Delta\bd(\bpo;t,\bb)\rangle \\
&\qquad\qquad  + \frac{\Delta^2t}{2}\normsq{\bd(\bpo;t,\bb)} \\
&= V(\bpo;t,\bb) + \Delta \langle\bs, \bd(\bpo;t,\bb)\rangle - \frac{\Delta^2t}{2}\normsq{\bd(\bpo;t,\bb)}
\end{alignedat}
\]
Using Equation~\eqref{eq:V_formula_descent} in the previous inequality and simplifying yields
\begin{equation}\label{eq:app_part4_1}
(t \Delta - 1)\normsq{\bd(\bpo;t,\bb)} \geqslant \langle\bs,\bd(\bpo;t,\bb)\rangle.
\end{equation}
Next, we use the inclusion $\mathcal{E}(\bpo + \Delta\bd(\bpo;t,\bb)) \subset \epo$, which was proven in Part 4, and subdifferentiability to find
\[
\begin{alignedat}{1}
V(\bpo;t,\bb) &\geqslant V(\bpo + \Delta\bd(\bpo;t,\bb);t,\bb) + \langle\bs,\bpo - (\bpo + \Delta\bd(\bpo;t,\bb)) \rangle \\
&\qquad \qquad + \frac{t}{2}\normsq{\bpo - (\bpo + \Delta\bd(\bpo;t,\bb))}\\
&= V(\bpo + \Delta\bd(\bpo;t,\bb);t,\bb) -\Delta\langle\bs,\bd(\bpo;t,\bb) \rangle + \frac{\Delta^{2}t}{2}\normsq{\bd(\bpo;t,\bb)}.\\
\end{alignedat}
\]
Using Equation~\eqref{eq:V_formula_descent} and substituting in the above yields
\begin{equation}\label{eq:app_part4_2}
\langle\bs,\bd(\bpo;t,\bb) \rangle \geqslant (t \Delta - 1)\normsq{\bd(\bpo;t,\bb)},
\end{equation}
Combining inequalities~\eqref{eq:app_part4_1} and~\eqref{eq:app_part4_2} yields the equality
\begin{equation}\label{eq:app_part4_3}
\langle\bd(\bpo;t,\bb),(1-t \Delta)\bd(\bpo;t,\bb) + \bs\rangle = 0.
\end{equation}

Finally, we use the projection characterization~\eqref{eq:char_projections} with 
\[
\bx = \boldsymbol{0},\, C = \partial_{\bp}V(\bpo + \Delta\bd(\bpo;t,\bb);t,\bb) \ \text{and} \ \by = -(1-t\Delta)\bd(\bpo;t,\bb), 
\]
and use~\eqref{eq:app_part4_3} to get the inequality
\begin{equation}\label{eq:app_part4_4}
\langle\bs, (1-t\Delta)\bd(\bpo;t,\bb) + \bs\rangle \leqslant 0.
\end{equation}
However, multiplying~\eqref{eq:app_part4_3} by $(1-t \Delta)$ and adding it to~\eqref{eq:app_part4_4} yields
\[
\begin{alignedat}{1}
\langle(1-t \Delta)\bd(\bpo;t,\bb) + \bs,(1-t \Delta)\bd(\bpo;t,\bb) + \bs\rangle = \normsq{(1-t\Delta)\bd(\bpo;t,\bb) + \bs} \leqslant 0.
\end{alignedat}
\]
We deduce $\bs = -(1-t \Delta)\bd(\bpo;t,\bb)$, that is, 
\[
-\proj_{\partial_{\bp}V(\bpo+\Delta\bd(\bpo;t,\bb);t,\bb)}(\boldsymbol{0}) \equiv \bd(\bpo+\Delta\bd(\bpo;t,\bb);t,\bb) = (1-t \Delta)\bd(\bpo;t,\bb).
\]


\subsection{Proof of Lemma~\ref{lem:tech_perturbation_I}} \label{app:tech_perturbation_I}
Use Proposition~\ref{lem:projection} with hyperparameter $\tzero + \delta_{0}$ to get
\[
\bd(\bpo;\tzero + \delta_{0},\bb) = \matr{A}\matr{D}(\bpo)\hat{\bu}(\bpo;\tzero + \delta_{0}, \bb) - \bb -(\tzero + \delta_{0})\bpo,
\]
where 
\[
\begin{alignedat}{1}
&\hat{\bu}(\bpo;\tzero + \delta_{0}, \bb) \in \argmin_{\bu \in \Rn} \normsq{\matr{A}\matr{D}(\bpo)\bu  - \bb -(\tzero + \delta_{0})\bpo} \\
&\qquad\qquad \text{subject to} \ 
	\begin{cases}
		\bu_{\epo} \geqslant \boldsymbol{0} \\
		\bu_{\ecpo} = \boldsymbol{0}.
	\end{cases}
\end{alignedat}
\]
Now let $\bv \in \Rm$ and use the change of variables $\bu = \hat{\bu}(\bpo;\tzero,\bb) + \bv$. The constraints of the problem become $\bv_{\epo} \geqslant -\hat{\bu}_{\epo}(\bpo;\tzero,\bb)$ and $\bv_{\ecpo} = \boldsymbol{0}$, while the objective function becomes
\[
\begin{alignedat}{1}
&\normsq{\matr{A}\matr{D}(\bpo)\bu  - \bb  -(\tzero + \delta_{0})\bpo} \\
	&\qquad \qquad \qquad = \normsq{\matr{A}\matr{D}(\bpo)\hat{\bu}(\bpo;\tzero,\bb) + \matr{A}\matr{D}(\bpo)\bv - \bb -(\tzero + \delta_{0})\bpo} \\
&\qquad \qquad \qquad = \normsq{\left(\matr{A}\matr{D}(\bpo)\hat{\bu}(\bpo;\tzero,\bb) -\bb - t\bpo\right) +\matr{A}\matr{D}(\bpo)\bv - \delta_{0}\bpo} \\
&\qquad \qquad \qquad = \normsq{\bd(\bpo;\tzero,\bb) + \matr{A}\matr{D}(\bpo)\bv - \delta_{0}\bpo},
\end{alignedat}
\]
From this, we obtain~\eqref{eq:descent_direction_change_in_t} and~\eqref{eq:nnls_problem_v}. Next, set $\bd(\bpo;\tzero,\bb) = \boldsymbol{0}$, $\delta_{0} = t - \tzero$ with $t \in [0,\tzero]$ and factor out the term $(1-t/\tzero)$ outside the optimization problem to obtain~\eqref{eq:descent_direction_change_in_t_2} and~\eqref{eq:nnls_problem_v_2}.


\subsection{Proof of Proposition~\ref{prop:generalized_homotopy}} \label{app:generalized_homotopy}
First, we invoke Lemma~\ref{lem:tech_perturbation_I}(ii) with 
\[
\bpo = \bp^{s}(\tzero,\bb), \ \hat{\bu}(\bpo;\tzero,\bb) = \matr{D}(\bp^{s}(\tzero,\bb))\bx^{s}(\tzero,\bb),
\]
and use the optimality conditions $\bd(\bp^{s}(\tzero,\bb);\tzero,\bb) = \boldsymbol{0}$ to simplify formula~\eqref{eq:descent_direction_change_in_t} to
\begin{equation}\label{eq:app_dd_linear_perturbations}
\bd(\bp^{s}(\tzero,\bb);t, \bb) = (1-t/\tzero)(\matr{A}\matr{D}(\bp^{s}(\tzero,\bb))\hat{\bv}(\tzero,t,\bb) + \tzero(\bp^{s}(\tzero,\bb))),
\end{equation}
where
    \begin{equation}\label{eq:app_nnls_linear_perturbations}
    \begin{alignedat}{1}
        &\hat{\bv}(\tzero, t,\bb) \in \argmin_{\bv \in \Rn} \normsq{\matr{A}\matr{D}(\bp^{s}(\tzero,\bb))\bv + t_{0}(\bp^{s}(\tzero,\bb))} \\
        &\text{subject to} \
            \begin{cases}
                &\bv_{j} \geqslant -|\bx^{s}_{j}(\tzero,\bb)|/(1-t/\tzero) \ \text{if $j \in \eqset(\bp^{s}(\tzero,\bb))$ and $\bx^{s}_{j}(\bb,\tzero) \neq 0$}  \\
                &\bv_{j} \geqslant 0 \ \text{if $j \in \eqset(\bp^{s}(\tzero,\bb))$ and $\bx^{s}_{j}(\bb,\tzero) = 0$,} \\
                &\bv_{\eqcset(\bp^{s}(\tzero,\bb))} = \boldsymbol{0}.
            \end{cases}
        \end{alignedat}
	\end{equation}	
At $t = \tzero$, problem~\eqref{eq:app_nnls_linear_perturbations} reduces to~\eqref{eq:nnls_problem_v_3} with the identification $\hat{\bv}(\tzero,\tzero,\bb) \equiv \hat{\bv}^{s}(\tzero,\bb)$.

We will now identify how the solution $\hat{\bv}^{s}(\tzero,\bb)$ behaves as we decrease $\tzero$. There are two potential sources of changes: the set of constraints in~\eqref{eq:app_nnls_linear_perturbations} and the minimal selection principle via the evolution rule~\eqref{eq:identity_dd} in Proposition~\ref{prop:descent_direction}. We turn to these two sources in turn.

First, consider the constraints in problem~\eqref{eq:app_nnls_linear_perturbations}:
\[
\text{$\bv_{j} \geqslant -\infty$ if $j \in \mathcal{E}(\bp^{s}(\tzero,\bb))$ and $|\bx^{s}(\tzero,\bb)| \neq \boldsymbol{0}$}.
\]
By assumption that $\tzero > 0$ and continuity, there is some $\epsilon_{0} > 0$ such that $0 \leqslant \tzero - \epsilon_0$,  $t \in [\tzero-\epsilon_{0},\tzero]$, and $\hat{\bv}^{s}(\tzero,\bb)$ remains the solution to problem~\eqref{eq:app_nnls_linear_perturbations}. It will remain the same as $\epsilon_{0}$ decreases until either $\epsilon_{0} = \tzero$ or there exists $j \in \mathcal{E}(\bp^{s}(\tzero,\bb))$ with $|\bx^{s}_{j}(\tzero,\bb)| \neq \boldsymbol{0}$ such that $\hat{\bv}^{s}_{j}(\tzero,\bb) = -|\bx^{s}_{j}(\tzero,\bb)|/(1-t/\tzero)$, that is, a constraint becomes satisfied with equality. In the latter case, we can rearrange this expression to obtain
	\[
	t = \tzero\left(1 - \frac{|\bx^{s}_{j}(\tzero,\bb)|}{|\hat{\bv}^{s}_{j}(\tzero,\bb)|}\right).
	\]
	The first index $j \in \mathcal{E}(\bp^{s}(\tzero,\bb))$ with $\bx^{s}_{j}(\tzero,\bb) \neq 0$ and $\hat{\bv}^{s}(\tzero,\bb) \leqslant -|\bx^{s}_{j}(\tzero,\bb)|$, if it exists, is the one whose ratio $|\bx^{s}_{j}(\tzero,\bb)|/|\hat{\bv}^{s}_{j}(\tzero,\bb)|$ is minimized. Hence
	\[
	T_{-}(\tzero,\bb,\hat{\bv}^{s}) \coloneqq \tzero \left(1 - \inf_{\substack{j \in \eqset(\bp^{s}(\tzero,\bb)) \\ \bx^{s}_{j}(\tzero,\bb) \neq 0 \\ \hat{\bv}_{j}^{s}(\tzero,\bb) \leqslant -|\bx_{j}^{s}(\tzero,\bb)|}} \frac{|\bx_{j}^{s}(\tzero,\bb)|}{|\hat{\bv}^{s}_{j}(\tzero,\bb)|}\right)
	\]
	is the smallest number lesser than $\tzero$ for which $\hat{\bv}^{s}(\tzero,\bb)$ solves problem~\eqref{eq:app_nnls_linear_perturbations}.

Now, consider the descent direction~\eqref{eq:app_dd_linear_perturbations}, suppose $t \in [\max(0,T_{-}(\tzero,\bb,\hat{\bv}^{s})),\tzero]$ and let
\[
\bxi^{s}(\tzero,\bb) \coloneqq \matr{A}\matr{D}(\bp^{s}(\tzero,\bb))\hat{\bv}^{s}(\tzero,\bb) + \tzero(\bp^{s}(\tzero,\bb))
\]
so as to write
\[
\bd(\bp^{s}(\tzero,\bb); t, \bb) = (1-t/\tzero)\bxi^{s}(\tzero,\bb).
\]
The descent direction depends linearly on $t$, and so the corresponding maximal descent time $\Delta_{*}(\bp^{s}(\tzero,\bb);t,\bb)$ in Proposition~\ref{prop:descent_direction}(i) is inversely proportional to $t$. More precisely:
\begin{equation}\label{eq:app_mdt_decomposition}
\Delta_{*}(\bp^{s}(\tzero,\bb);t,\bb) = C^{s}(\tzero,\bb)/(1-t/\tzero).
\end{equation}
where
\[
    C^{s}(\tzero, \bb) \coloneqq
     \min_{j \in \{1,\dots,n\}} \left\{\frac{\sign{\left(\langle \matr{D}(\bp^{s}(\tzero,\bb))\Atop\bxi^{s}(\tzero,\bb),\be_j\rangle\right)} - \langle \matr{D}(\bp^{s}(\tzero,\bb))\Atop\bp^{s}(\tzero,\bb),\be_j\rangle)}{\langle\matr{D}(\bp^{s}(\tzero,\bb)) \Atop\bxi^{s}(\tzero,\bb),\be_j\rangle} \right\}.
\]
Furthermore, the evolution rule from Proposition~\ref{prop:descent_direction}(v) yields
\[
\bd(\bp^{s}(\tzero,\bb) + \Delta(1-t/\tzero)\bxi^{s}(\tzero,\bb); t, \bb) = (1-t\Delta)(1-t/\tzero)\bxi^{s}(\tzero,\bb).
\]
We now seek the smallest nonnegative number $T_{+}(\tzero,\bb) \leqslant \tzero$ in terms of $\Delta_{*}(\bp^{s}(\tzero,\bb);t,\bb)$ for which 
\[
1-t\Delta_{*}(\bp^{s}(\tzero,\bb);t,\bb) = \boldsymbol{0}
\]
for every $t \in [T_{+}(\tzero,\bb),\tzero]$. Equation~\eqref{eq:app_mdt_decomposition} gives
\[
\begin{alignedat}{1}
1-t\Delta_{*}(\bp^{s}(\tzero,\bb);t,\bb) = \boldsymbol{0} &\iff 1-tC^{s}(\tzero,\bb)/(1-t/\tzero) = 0 \\
&\iff 1 - t/\tzero + tC^{s}(\tzero,\bb) = 0\\
&\iff t = \tzero/(1+\tzero C^{s}(\tzero,\bb)).
\end{alignedat}
\]
The critical value is
\[
T_{+}(\tzero,\bb) \coloneqq \frac{\tzero}{1+\tzero C^{s}(\tzero,\bb)}.
\]
Hence letting
\[
t_{1} \coloneqq \max(T_{-}(\tzero,\bb,\hat{\bv}^{s}),T_{+}(\tzero,\bb)),
\]
for every $t \in (t_{1},\tzero]$ there exists some $\Delta \in [0,\Delta_{*}(\bp^{s}(\tzero,\bb);t,\bb))$ for which $1 - t\Delta = 0$. Note that $0 \leqslant t_{1} < \tzero$ since $0 \leqslant T_{+}(\tzero,\bb) < \tzero$ and $T_{-}(\tzero,\bb,\hat{\bv}^{s}) < \tzero$. In particular, we have that $t_{1} = 0 \implies T_{+}(\tzero,\bb) = 0$ and
\[
T_{+}(\tzero,\bb) = 0 \iff C^{s}(\tzero, \bb) = +\infty \iff \Atop\bxi^{s}(\tzero,\bb) = \boldsymbol{0}.
\]

Taken together, we arrive at the following result: For every $t \in (t_{1},\tzero]$, we have
\[
\bd\left(\bp^{s}(\tzero,\bb) + \left(\frac{1}{t}-\frac{1}{\tzero}\right)\bxi^{s}(\tzero,\bb); t, \bb \right) = \boldsymbol{0}.
\]
Using the optimality conditions~\eqref{eq:bpdn_full_oc} and Lemma~\ref{lem:discontinuity}, we conclude
\[
\bp^{s}(t,\bb) = \begin{cases}
 		&\bp^{s}(\tzero,\bb) + \left(\frac{1}{t} - \frac{1}{\tzero}\right)\bxi^{s}(\tzero,\bb) \quad \text{if $t_{1} > 0$}, \\
 		&\bp^{s}(\tzero,\bb) \quad \text{otherwise,}
 		\end{cases}
\]
is the solution to~\eqref{eq:lasso_dual} at hyperparameter $t$ and data $\bb$ on $[t_{1},\tzero]$. Furthermore,
\[
\begin{alignedat}{1}
	t\bp^{s}(t,\bb) &= \left(\frac{t}{\tzero}\right)\tzero\bp^{s}(\tzero,\bb) + \left(1-\frac{t}{\tzero}\right)\bxi^{s}(\tzero,\bb)\\
	&= \tzero\bp^{s}(\tzero,\bb) - \left(1 - \frac{t}{\tzero}\right)\tzero\bp^{s}(\tzero,\bb) \\
	&\qquad\qquad + \left(1 - \frac{t}{\tzero}\right)(\matr{A}\matr{D}(\bp^{s}(\tzero,\bb)) + \tzero(\bp^{s}(\tzero,\bb)))\\
	&= \tzero\bp^{s}(\tzero,\bb) + \left(1-\frac{t}{\tzero}\right)\left(\matr{A}\matr{D}(\bp^{s}(\tzero,\bb))\hat{\bv}^{s}(\tzero,\bb) \right) \\
	&= \matr{A}\left[\bx^{s}(\tzero,\bb) + \left(1 - \frac{t}{\tzero}\right)\matr{D}(\bp^{s}(\tzero,\bb))\hat{\bv}^{s}(\tzero,\bb)\right] - \bb,
\end{alignedat}
\]
and so we deduce
\[
\bx^{s}(t,\bb) = \bx^{s}(\tzero,\bb) + \left(1-\frac{t}{\tzero}\right)\matr{D}(\bp^{s}(\tzero,\bb))\hat{\bv}^{s}(\tzero,\bb)
\]
is the primal solution to~\eqref{eq:lasso_primal} at hyperparameter $t$ and data $\bb$ on $[t_{1},\tzero]$.

\bibliography{proj-bib}
\end{document}